\documentclass[11pt]{amsart}
\usepackage{amscd,amsxtra,amssymb,mathrsfs, bbm}
\usepackage{stmaryrd}
\usepackage{amscd,amsxtra,amssymb, bbm, url}
\usepackage[new]{old-arrows}

\usepackage{geometry}
        {\begin{quotation}\begin{center}\begin{em}}
        {\par\end{em}\end{center}\end{quotation}}

\newtheorem{theorem}{Theorem}[section]

\newtheorem{lemma}[theorem]{Lemma}
\newtheorem{proposition}[theorem]{Proposition}

\theoremstyle{definition}
\newtheorem{definition}[theorem]{Definition}
\newtheorem{remark}[theorem]{Remark}
\newtheorem{example}[theorem]{Example}

\DeclareMathOperator{\SL}{\mathsf{SL}}

\newcommand{\id}{\mathrm{id}}

\DeclareMathOperator{\Coh}{\mathsf{Coh}}
\DeclareMathOperator{\Qcoh}{\mathsf{QCoh}}
\DeclareMathOperator{\Tor}{\mathsf{Tor}}

\DeclareMathOperator{\VB}{\mathsf{VB}}

\DeclareMathOperator{\Pic}{\mathsf{Pic}}

\DeclareMathOperator{\Hom}{\mathsf{Hom}}

\DeclareMathOperator{\Ext}{\mathsf{Ext}}

\DeclareMathOperator{\Aut}{\mathsf{Aut}}
\DeclareMathOperator{\End}{\mathsf{End}}

\DeclareMathOperator{\Spec}{\mathsf{Spec}}

\DeclareMathOperator{\dX}{\mathfrak{X}}

\DeclareMathOperator{\dY}{\mathfrak{Y}}

\DeclareMathOperator{\idm}{\mathfrak{m}}

\newcommand{\kk}{\mathbbm{k}}
\newcommand{\KK}{\mathbbm{K}}
\newcommand{\CC}{\mathbb{C}}
\newcommand{\RR}{\mathbb{R}}
\newcommand{\NN}{\mathbb{N}}
\newcommand{\ZZ}{\mathbb{Z}}
\newcommand{\LL}{\mathbb{L}}

\newcommand{\FF}{\mathbb{F}}

\newcommand{\EE}{\mathbb{E}}

\newcommand{\ff}{\mathbbm{f}}
\renewcommand{\gg}{\mathbbm{g}}

\newcommand{\XX}{\mathbb{X}}
\newcommand{\YY}{\mathbb{Y}}

\newcommand{\HH}{\mathsf{H}}

\renewcommand{\HH}{\mathbb{H}}
\newcommand{\PP}{\mathbbm{P}}

\newcommand{\sHom}{\mbox{\it{Hom}}}
\newcommand{\sExt}{\mbox{\it{Ext}}}

\newcommand{\kA}{\mathcal{A}}

\newcommand{\kC}{\mathcal{C}}
\newcommand{\kE}{\mathcal{E}}
\newcommand{\kF}{\mathcal{F}}
\newcommand{\kG}{\mathcal{G}}
\newcommand{\kH}{\mathcal{H}}
\newcommand{\kI}{\mathcal{I}}

\newcommand{\kO}{\mathcal{O}}

\newcommand{\kP}{\mathcal{P}}

\newcommand{\kR}{\mathcal{R}}
\newcommand{\kK}{\mathcal{K}}

\newcommand{\kS}{\mathcal{S}}
\newcommand{\kX}{\mathcal{X}}

\newcommand{\kZ}{\mathcal{Z}}

\newcommand{\lar}{\longrightarrow}

\def\sD{\mathsf D} 
\def\sE{\mathsf E} 
\def\sF{\mathsf F} 
\def\sG{\mathsf G} \def\sT{\mathsf T}
 
\def\sI{\mathsf I}

\def\kron#1#2{\xymatrix@C=2em{{#1}\ar@/^3pt/[r]\ar@/_3pt/[r]&{#2}}}
\def\bu{{\scriptscriptstyle\bullet}}

\usepackage{tikz}
\usetikzlibrary{calc, arrows, positioning, shapes, fit, matrix, decorations}
\usetikzlibrary{decorations.shapes, decorations.pathreplacing}

\tikzset{
  decorate with/.style={decorate,decoration={shape backgrounds,shape=#1,shape size=1.5mm}},
   deco/.style={decorate with=dart},
   ordi/.style={draw,-stealth,  thick},
   conj/.style={dashed, draw, thick},
   ve/.style={circle, draw, thick, fill=blue!20, inner sep=1pt, outer sep=2pt, minimum size=7pt},
    dot/.style={fill=blue!10,circle,draw, inner sep=1pt, minimum size=5pt},
  dv/.style={star,star points=5,
star point ratio=2, draw, thick, fill=green!20, inner sep=1pt,outer sep=2pt,minimum size=7pt}
}

\tikzset{
    tbl5 nodes/.style={
        rectangle,
        execute at begin node=$,
       execute at end node=$,
       fill=blue!5,
        align=center,
        text depth=0.5ex,
        text height=2ex,
        inner xsep=0pt,
        outer sep=0pt,
           },
    tbl5/.style={
        matrix of nodes,
        row sep=-\pgflinewidth,
        column sep=-\pgflinewidth,
        nodes={
            tbl5 nodes
        },
        execute at empty cell={\node[draw=none]{};}
    }
  }

\input xy
\xyoption{all}

\title[Exceptional curves and real curve orbifolds]{Exceptional hereditary curves and real curve orbifolds}

\author{Igor Burban}
\address{
Paderborn University
}
\email{burban@math.uni-paderborn.de}

\dedicatory{Dedicated to Yuriy Drozd
on the occasion of his 80th birthday}

\begin{document}

\begin{abstract}
In this paper, we elaborate the theory of exceptional hereditary curves over arbitrary fields. In particular, we study the category of equivariant coherent sheaves on a regular projective curve whose quotient curve has genus zero and prove  existence of a tilting object in this case. We also establish  a link between wallpaper groups and real hereditary curves, providing details to an old observation made by Helmut Lenzing.  
\end{abstract}

\maketitle

\section{Introduction}

Let $\kk$ be an arbitrary field. 
The categories  $\Coh(\XX)$ of coherent sheaves on  a non-commutative projective  hereditary curve $\mathbb{X} = (X, \kH)$ (where $X = (X, \kO)$ is a commutative regular projective   curve over $\kk$ and $\mathcal{H}$ is a sheaf of hereditary $\mathcal{O}$-orders) provide an important class of  $\kk$-linear $\Ext$-finite hereditary categories.  In the case when $X = \PP^1_{\kk}$ and $\kk = \bar{\kk}$ is algebraically closed, $\Coh(\XX)$ is equivalent to the category of coherent sheaves  on an appropriate  weighted projective line of Geigle and Lenzing and admits a tilting object \cite{GeigleLenzing}. In particular, there exist a finite dimensional $\kk$-algebra $\Sigma$ (which belongs to the class of so-called canonical algebras 
\cite{RingelTameAlgebras}) and an exact equivalence of derived categories
\begin{equation}\label{E:GLTilting}
 D^b\bigl(\Coh(\XX)\bigr) \lar D^b(\Sigma\mathsf{-mod}).
 \end{equation}
In the case of an arbitrary base field 
$\mathbbm{k}$, a hereditary projective curve $\XX$ is called \emph{exceptional} if its derived category $D^b\bigl(\Coh(\XX)\bigr)$ admits a tilting object. Dropping  the assumption $\kk = \bar{\kk}$ makes the theory of such curves 
significantly richer. Firstly, the underlying commutative curve $X$ can be an arbitrary Brauer--Severi curve.  Another reason for complications is caused by the fact that the  Brauer group $\mathsf{Br}\bigl(\kk(X)\bigr)$ of the function field $\mathbbm{k}(X)$ of $X$ is no longer zero and arithmetic phenomena start to play  an important role in the study of the category $\Coh(\XX)$. At this point let us mention that the Brauer class 
 $\eta_{\XX} = \bigl[\Gamma(X, \kK \otimes_\kO \kH)\bigr] \in \mathsf{Br}\bigl(\kk(X)\bigr)$ of an exceptional hereditary curve $\XX$ can not take arbitrary values. Moreover, $\XX$ is a weighted projective line if and only if $X = \PP^1_\kk$ and $\eta_{\XX} = 0$.  
 
The study of exceptional hereditary curves over arbitrary base fields was initiated by Lenzing in \cite{LenzingFirst}. However, the underlying hereditary abelian category $\Coh(\XX)$ was defined in an implicit way, without an involvement of sheaves of orders. Quoting for example \cite[page 415]{HappelReiten}:  ``Since there is at present no ``geometric'' definition available for coherent sheaves  on a weighted projective line over an arbitrary field, the formulation of our main result will be somewhat different from the formulation for an algebraically closed field $\kk$.''

A classification of $\kk$-linear $\Ext$-finite hereditary abelian categories   (see 
\cite{Happel, ReitenvandenBergh} for the case $\kk = \bar{\kk}$ and \cite{HappelReiten, LenzingReiten} for an arbitrary $\kk$) allowed one to define  exceptional hereditary curves in an ``axiomatic way'' by providing a list of characterizing properties of the category $\Coh(\XX)$. In this work, we give a further elaboration of this theory, starting with a ringed space $\XX=(X, \kH)$ itself as a primary object.  

The first main result of this paper is Theorem \ref{T:TiltingMain} which gives a straightforward construction  of a tilting complex in the derived category  $D^b\bigl(\Coh(\XX)\bigr)$
for a complete  hereditary curve $\XX$ of a special type. This allows one to prove a generalization of the equivalence (\ref{E:GLTilting}) in the case of an arbitrary field $\kk$.

A natural class of examples of exceptional heredirary curves arise from finite group actions. Let $Y$ be a complete regular curve over $\kk$ and $G \subset \Aut_{\kk}(Y)$ be a finite group such that
$\mathsf{gcd}\bigl(|G|, \mathsf{char}(\kk)\bigr) = 1$ and the quotient $X = Y/G$ is a curve of genus zero. Then there exists a hereditary curve $\XX = Y \hspace{-0.75mm}\sslash \hspace{-0.75mm}G = (X, \kH)$ such that $\Coh^G(Y) \simeq \Coh(\XX)$, where 
 $\Coh^G(Y)$ is the category of $G$-equivariant coherent sheaves on $Y$. This result is well-known but we elaborate its proof in Proposition \ref{P:EquivNonComm}. Then we show that all such $\XX$ are exceptional with $\eta_{\XX} = 0$ (see Theorem \ref{T:Equivariant}), extending results of \cite{Orbifolds} on the case of an arbitrary base field $\kk$; see also \cite{LenzingFirst, GeigleLenzing, Kirillov}. 

Wallpaper groups lead to a very interesting class  of finite group actions \emph{over} $\RR$ on \emph{complex} elliptic  curves, what allows one to make  a link  to the so-called real tubular curves.  This striking observation was made by Lenzing many years ago \cite{LenzingLastTalk}, although the underlying  details were never published. This gap in the literature is filled by  Theorem \ref{T:Wallpapers}. Namely, to any wallpaper group $W$ one can attach a hereditary curve $\XX_W$ and in 13 cases out of 17 the corresponding derived category $D^b\bigl(\Coh(\XX_W)\bigr)$ admits a tilting object, whose endomorphism algebra $\Sigma_W$ is a tubular canonical algebra and whose type can be read off the orbifold description of the group $W$;    see also Remark \ref{R:LenzingKussin} for a different approach.

\smallskip
\noindent
\emph{Acknowledgement}. This work was partially supported by the German Research Foundation
SFB-TRR 358/1 2023 -- 491392403. I am grateful to Dirk Kussin, Daniel Perniok and Charly Schwabe  for fruitful discussions of results related to this work.

\section{Hereditary orders}
We begin by recalling the notion of a classical order and its properties.  

\begin{definition}\label{D:Order}
Let $O$ be  an excellent  reduced  equidimensional ring of {Krull dimension one}  and
$K := \mathsf{Quot}(O)$ be the corresponding total ring of fractions. 
An $O$-algebra $A$ is an $O$-\emph{order} if the following conditions are fulfilled:
\begin{itemize}
\item $A$ is a finitely generated torsion free $O$-module.
\item $A_K:= K \otimes_O A$ is a  semi-simple $K$-algebra, having finite length as a $K$-module.
\end{itemize}
\end{definition}

\noindent
Let $O$ be as above, $O' \subseteq O$ be a subring such that the corresponding ring extension is finite and $A$ be an $O$-algebra. Then $A$ is an $O$-order if and only if $A$ is an $O'$-order.
Moreover, if $K' := \mathsf{Quot}(O')$ then we have: $A_K \cong A_{K'}$; see for instance \cite[Lemma 2.8]{bdnpdalcurves}.

\begin{definition}\label{D:orders} Let $A$ be a ring.
\begin{itemize}
\item $A$ is  a  \emph{classical order} (or just an \emph{order}) provided  its center $O = Z(A)$ is a reduced excellent  equidimensional ring  of Krull dimension one, and $A$ is an $O$-order.
\item Let
$K:= \mathsf{Quot}(O)$. Then $A_K:= K \otimes_O A$ is called the \emph{rational envelope} of $A$.
\item A ring $\widetilde{A}$ is called an \emph{overorder} of $A$ if $A \subseteq \widetilde{A} \subset A_K$ and $\widetilde{A}$ is finitely generated as (a left) $A$-module.
\item An order $A$ is called \emph{maximal} if it has no proper overorders. 
\end{itemize}
\end{definition}

\noindent
Note that for any overorder $\widetilde{A}$ of $A$,  the map $K \otimes_O \widetilde{A} \lar A_K$ is automatically an isomorphism. Hence,
$A_K = \widetilde{A}_K$ and $\widetilde{A}$ is an order over $O$.

\begin{lemma}\label{L:HeredOrders} Let $H$ be an order   and  $O = Z(H)$ be its center. Then the following results are true.
\begin{enumerate}
\item[(a)] Assume that $H$ is \emph{hereditary} (i.e.~$\mathsf{gl.dim}(H) =1$).
Then $O \cong O_1 \times \dots \times O_r$, where  $O_i$ is a   Dedekind domain for all $1\le i \le r$.
\item[(b)] Suppose that $O$ is semilocal. Let   $J$  be the Jacobson radical of $H$ and
$\widehat{H} =
 \varprojlim\limits_{k} \bigl(H/J^k\bigr)$
be the completions of $H$. Then $H$ is hereditary  if and only if $\widehat{H}$ is  hereditary.
\end{enumerate}
\end{lemma}

\smallskip
\noindent
\emph{Proofs} of all these results can be for instance found in \cite{ReinerMO}. \qed

\smallskip
\noindent
 Let $O$ be a complete discrete valuation ring,  $A$ be a maximal order with center $O$ and $J$ be the Jacobson radical of $A$. We chose an element $w \in J$ such that 
 $J = Aw = w A$; see \cite[Theorem 18.7]{ReinerMO} for the existence of such $w$.
 For any sequence of natural numbers
$\vec{p} = \bigl(p_1, \dots, p_r\bigr)$, consider the $O$-algebra
\begin{equation*}
H(A, \vec{p}) :=
\left[
\begin{array}{ccc|ccc|c|ccc}
A & \dots & A & J & \dots & J & \dots & J & \dots &  J  \\
\vdots & \ddots & \vdots & \vdots &  & \vdots &  & \vdots & & \vdots \\
A & \dots & A & J & \dots & J & \dots & J & \dots &  J  \\
\hline
A & \dots & A & A & \dots & A & \dots & J & \dots &  J  \\
\vdots &  & \vdots & \vdots  & \ddots & \vdots &  & \vdots & & \vdots \\
A & \dots & A & A & \dots & A & \dots & J & \dots &  J  \\
\hline
\vdots &  & \vdots & \vdots & & \vdots & \ddots & \vdots &   & \vdots\\
\hline
A & \dots & A & A & \dots & A & \dots & A & \dots &  A  \\
\vdots &  & \vdots &  \vdots &  & \vdots &  & \vdots & \ddots & \vdots \\
A & \dots & A & A & \dots & A & \dots & A & \dots &  A  \\
\end{array}
\right] = \left[
\begin{array}{cccc}
A & J & \dots & J \\
A & A & \dots & J\\
\vdots & \vdots & \ddots & \vdots \\A  & A & \dots & A\\
\end{array}
\right]^{\underline{(p_1, \dots, p_r)}}
\end{equation*}
where the size
of the $i$-th diagonal block  is $(p_i \times p_i)$ for each $1 \le i \le r$.

\begin{theorem} The following results are true. 
\begin{enumerate}
\item[(i)] $H(A, \vec{p})$ is a hereditary order, whose center is isomorphic to $O$.
\item[(ii)] Let $A'$ be another maximal order and $\vec{p}' \in \NN^{s}$. Then $H(A, \vec{p}) \cong H(A', \vec{p}')$ if and only if $A \cong A'$, $r = s$ and $\vec{p}'$ is a cyclic shift of $\vec{p}$.
\item[(iii)] Let $H$ be a hereditary order, whose center is isomorphic to $O$. Then $H \cong 
H(A, \vec{p})$ for some maximal order $A$ and a vector $\vec{p} \in \NN^r$ for some $r \in \NN$.
\item[(iv)] We have the following description of the Jacobson radical of $H = H(A, \vec{p})$: 
\begin{equation*}
\mathsf{rad}(H) = 
\left[
\begin{array}{cccc}
J & J & \dots & J \\
A & J & \dots & J\\
\vdots & \vdots & \ddots & \vdots \\A  & A & \dots & J\\
\end{array}
\right]^{\underline{(p_1, \dots, p_r)}}.
\end{equation*}
In particular, we have:
$$
H/\mathsf{rad}(H) \cong M_{p_1}(D) \times \dots \times M_{p_r}(D),
$$
where $D = A/J$ is the residue skew field of $A$. 
\item[(v)] Let $\vec{e} := (1, \dots, 1) \in \NN^r$. Then the orders $H(A, \vec{p})$ and
$$H_r(A) := H(A, \vec{e}) = \left[
\begin{array}{cccc}
A & J & \dots & J \\
A & A & \dots & J\\
\vdots & \vdots & \ddots & \vdots \\A  & A & \dots & A\\
\end{array}
\right]
$$ are Morita equivalent. 
\end{enumerate}
\end{theorem}

\smallskip
\noindent
\emph{Proofs} of all these results can be for instance found in \cite{Harada, Harada2} as well as in \cite{ReinerMO}. \qed

\begin{remark}
In what follows, the hereditary order $H = H(A, \vec{p})$  will be called \emph{standard}. Moreover, the following statements are true. 
\begin{enumerate}
\item[(i)] There are precisely $r$ pairwise non-isomorphic indecomposable projective left $H$-modules:
\begin{equation}\label{E:IndecompProjectives}
P_1 = \left[
\begin{array}{c}
A \\
A \\
\vdots \\
A \\
\end{array}
\right]^{\underline{(p_1, \dots, p_r)}}
P_2 = \left[
\begin{array}{c}
J \\
A \\
\vdots \\
A \\
\end{array}
\right]^{\underline{(p_1, \dots, p_r)}}
\quad \mbox{\rm \dots} \quad 
P_r = 
\left[
\begin{array}{c}
J \\
J \\
\vdots \\
A \\
\end{array}
\right]^{\underline{(p_1, \dots, p_r)}}.
\end{equation}
\item[(ii)] Next, there are exactly $r$ pairwise non-isomorphic simple left $H$-modules $S_1, \dots, S_r$, whose minimal  projective resolutions are 
\begin{equation}\label{P:ResolutionsSimples}
 0 \lar P_{j+1} \stackrel{\varepsilon_j}\lar P_j \lar S_j \lar 0 \quad  \mbox{\rm for} \quad 1 \le j \le r.
\end{equation}
For $1 \le j < r$ the morphism $\varepsilon_j$ is just the natural inclusion, whereas $P_{r+1} = P_1$ and $\varepsilon_r$ is given by the right multiplication with the chosen generator $w \in J$. 
\item[(iii)] Let $1 \le i, j \le r$. It is clear that
\begin{equation*}
\Hom_{H}(S_i, S_j) \cong \left\{
\begin{array}{cc}
D^\circ & \mbox{\rm if} \; i = j \\
0 & \mbox{\rm otherwise},
\end{array}
\right.
\end{equation*}
where $D^\circ$ is the opposite ring of $D$. Moreover, 
\begin{equation}
\Ext^1_{H}(S_i, S_j) \cong \left\{
\begin{array}{cc}
D^\circ & \mbox{\rm if} \; j = i+1\\
0 & \mbox{\rm otherwise},
\end{array}
\right.
\end{equation}
where $S_{r+1} = S_1$. 
\item[(iv)] Let $A = \kk\llbracket z\rrbracket$. Then $H_r(A)$ is isomorphic to the arrow completion $\widehat{\kk\bigl[\vec{C}_r\bigr]}$ of the path algebra of the cyclic quiver $\vec{C}_r$:
\begin{equation}\label{E:Cyclicquiver}
\begin{array}{c}
\xymatrix{ & \stackrel{r}{\circ} \ar[ld]_-{a_1} & &\ar[ll]_-{a_r} \circ & \\
\stackrel{1}\circ  \ar[rd]_-{a_2} &         &   & \vdots  & \ar[lu] \circ\\
& \stackrel{2}\circ  \ar[rr] & & \circ \ar[ru] & \\
}
\end{array}
\end{equation}
\end{enumerate}
\end{remark}

Let $A \times A \stackrel{\kappa}\lar O$ be the pairing induced by the so-called \emph{reduced trace map} $A \stackrel{\mathit{tr}}\lar O$; see \cite[Section 9]{ReinerMO}. It is symmetric and invariant (i.e.~$\kappa(a, b) = \kappa(b, a)$ and 
$\kappa(ab, c) = \kappa(a, bc)$ for any $a, b, c \in A$). Moreover, it defines an isomorphism of ($A$--$A$)-bimodules
$$
A \lar \Omega_A := \Hom_O(A, O), \, a \mapsto \kappa(a, \,-\,).
$$
As a consequence, we have the following isomorphisms of ($H$--$H$)-bimodules:
$$
\Omega = \Omega_H := \Hom_O(H, O) \cong \Hom_A\bigl(H, \Hom_O(A, O)\bigr) \cong \Hom_A(H, A). 
$$
It follows that 
$$
\Omega \cong 
\left[
\begin{array}{cccc}
A & A & \dots & A \\
J^{-1} & A & \dots & A\\
\vdots & \vdots & \ddots & \vdots \\J^{-1}  & J^{-1} & \dots & A\\
\end{array}
\right]^{\underline{(p_1, \dots, p_r)}} 
$$
where $J^{-1} = A w^{-1} = w^{-1} A$ viewed as a subset of the rational hull of $A$.

Consider the functor  $\tau := \Omega \otimes_H \,-\,: H\mathsf{-mod} \lar H\mathsf{-mod}$. It is clear that
$$
\tau(P_1) \cong 
\left[
\begin{array}{c}
A \\
J^{-1} \\
\vdots \\
J^{-1} \\
\end{array}
\right]^{\underline{(p_1, \dots, p_r)}}
\cong 
\left[
\begin{array}{c}
J \\
A \\
\vdots \\
A \\
\end{array}
\right]^{\underline{(p_1, \dots, p_r)}}
 = P_2,
$$
where the last isomorphism is given by the right multiplication with $w$. In the same  vein, 
we have: $\tau(P_i) \cong P_{i+1}$ for all $1 \le i \le r$. Note that $\Omega$ is projective (hence flat) viewed as a right $H$-module. It follows that $\tau$ is an exact functor. Actually, $\tau$ is an auto-equivalence of $H\mathsf{-mod}$; see the discussion below. It follows from (\ref{P:ResolutionsSimples}) that $\tau(S_i) \cong S_{i+1}$ for all $1 \le i \le r$. 

\section{Exceptional hereditary curves}
Let $\kk$ be any field and $X$ be a  reduced quasi-projective equidimensional scheme of finite type over $\kk$ of Krull dimension one. Let $X_\circ$ be the set of closed points of $X$, $\kO$ be the structure sheaf of $X$, $\kK$ be its sheaf of rational functions and $\KK := \kK(X)$ be  the ring  of rational functions on $X$. We follow the terminology introduced in   \cite[Section 7]{BurbanDrozdMorita}.

\begin{definition} A  non-commutative curve over $\kk$ is a ringed space $\XX = (X, \kR)$, where $X$ is a commutative curve as above and $\kR$ is a sheaf of $\kO_X$-orders (i.e. $\kR(U)$ is an $\kO(U)$-order for any open affine subset $U \subseteq X$), which is coherent as a sheaf of $\kO_X$-modules.  Such  $\XX$ is called 
\begin{enumerate}
\item[(a)]  \emph{central} if $\kO_x$ is the center of $\kR_x$,
\item[(b)] \emph{homogeneous} (also called \emph{regular} in \cite{BurbanDrozdMorita}) if the order $\kR_x$ is maximal,
\item[(c)] \emph{hereditary} if  the order $\kR_x$ is hereditary
\end{enumerate}
for any $x \in X_\circ$. 
\end{definition}

\begin{remark}
Without loss of generality one may  assume $\XX = (X, \kR) $ to be central; see \cite[Remark 2.14]{BurbanDrozdMorita}. We call such $\XX$ \emph{complete} if $X$ is complete (i.e.~integral and proper (hence, projective))  over $\kk$. Then $\KK$ is a field and $\FF_{\XX} := \Gamma(X, \kK \otimes_\kO \kR)$ is a central simple algebra over $\KK$. Let $\eta := [\FF_{\XX}]$ be the corresponding class in the Brauer group $\mathsf{Br}(\KK)$ of  $\KK$.

We shall denote by $g(X)$ the genus of $X$. From now on, if not otherwise stated, all non-commutative curves over $\kk$ are assumed to be \emph{central and complete} and we shall frequently omit the term ``non-commutative'' when speaking about such $\XX$. 
\end{remark}

If $\XX = (X, \kR)$ is hereditary then $X$ is regular; see Lemma \ref{L:HeredOrders}.  Recall the following easy but fundamental fact due to Artin and de Jong \cite[Proposition 1.9.1]{ArtindeJong} (see also \cite[Proposition 2.9]{Spiess} and \cite[Corollary 7.9]{BurbanDrozdMorita}). 

\begin{theorem}\label{T:NCRegular} Let $X$ be a complete regular curve over $\kk$. Then for any $\eta \in \mathsf{Br}(\KK)$ there exists a 
homogeneous   curve $\XX = (X, \kR)$ such that $\bigl[\FF_\XX\bigr] = \eta$.  If $\XX' = (X', \kR')$ is another homogeneous  curve then the following statements are equivalent:
\begin{enumerate}
\item[(a)] The categories of coherent sheaves $\Coh(\XX)$ and $\Coh(\XX')$ are equivalent.
\item[(b)] There exists an isomorphism $X \stackrel{f}\lar X'$ such that $[\FF_{\XX}] = f^\ast\bigl([\FF_{\XX'}]\bigr) \in \mathsf{Br}(\KK)$.
\end{enumerate}
\end{theorem}

\begin{remark} In the above theorem, the ringed spaces $\XX$ and $\XX'$ need not be isomorphic even if we assume $\FF$ and $\FF'$ to be skew fields; see \cite[Remark 7.11]{BurbanDrozdMorita} and references therein.
\end{remark}

Let $\XX = (X, \kH)$ be a hereditary curve. The full subcategory of finite length objects of $\Coh(\XX)$ is denoted by $\Tor(\XX)$. Clearly,  it splits into a union of blocks:
\begin{equation}\label{E:TorsionSplitting}
\Tor(\XX) = \bigvee\limits_{x \in X_\circ} \Tor_x(\XX),
\end{equation}
where $\Tor_x(\XX)$ is equivalent to the category of finite length modules over the hereditary order $\kH_x$ for any $x \in X_\circ$.

We denote by $\VB(\XX)$ the full subcategory of the category $\Coh(\XX)$ consisting of locally projective objects, i.e.~those $\kE \in \Coh(\XX)$ for which each stalk $\kE_x$ is projective over $\kH_x$ for any $x \in X_\circ$. Similarly to the case of regular commutative curves, one can show that for any 
$\kF \in \Coh(\XX)$ there exist unique $\kE \in \VB(\XX)$ and $\kZ \in \Tor(\XX)$ such that
$\kF \cong \kE \oplus  \kZ$. 

\smallskip
\noindent
Consider the Serre quotient category $\Coh(\XX)/\Tor(\XX)$. 
Then the functor $$
\Gamma(X, \kK \otimes_\kH \,-\,): \Coh(\XX)/\Tor(\XX) \lar \FF_{\XX}\mathrm{-}\mathsf{mod}
$$
is an equivalence of categories. For any $\kF \in \Coh(\XX)$ we define its rank by the formula
$$
\mathsf{rk}(\kF) := \mathsf{length}_{\FF_{\XX}}\bigl(\Gamma(X, \kK \otimes_\kH \kF)\bigr).
$$
Objects of $\VB(\XX)$ of rank one are called \emph{line bundles}, the corresponding full subcategory of $\VB(\XX)$ is denoted by $\Pic(\XX)$. 

\begin{theorem}\label{T:Characteristic} Let $\XX = (X, \kH)$ be a  hereditary curve. Then the  following results  are true.
\begin{enumerate}
\item[(a)] $\Coh(\XX)$ is an $\Ext$-finite noetherian  hereditary $\kk$-linear  abelian category.  
\item[(b)] Let $\mathit{\Omega} = \mathit{\Omega}_{\XX} := \mathit{Hom}_{X}\bigl(\kH, \mathit{\Omega}_{X}\bigr)$, where $\mathit{\Omega}_X$ is the dualizing sheaf of $X$. Then 
\begin{equation}\label{E:ARTranslate}
\tau := \mathit{\Omega} \otimes_{\kH} \, -\, : \; \Coh(\XX) \lar \Coh(\XX)
\end{equation}
is an auto-equivalence of $\Coh(\XX)$. It restricts to auto-equivalences of its full subcategories $\VB(\XX)$, $\Tor(\XX)$ as well as $\Tor_x(\XX)$ for any $x \in X$.
\item[(c)] Moreover, for any $\kF, \kG \in \Coh(\XX)$ there are bifunctorial isomorphisms
\begin{equation}\label{E:ARDuality}
\Hom_{\XX}(\kF, \kG) \cong \Ext^1_{\XX}\bigl(\kG, \tau(\kF)\bigr)^\ast.
\end{equation}
\end{enumerate}
\end{theorem}

\noindent
\emph{Comment to the proof}. Properties of the functor $\tau$ follow from much more general results about dualizing complexes and Serre functors; see 
for example \cite[Theorem A.4]{NaeghvdBergh} and
\cite[Proposition 6.14]{YekutieliZhang}.

\begin{remark} The category of coherent sheaves $\Coh(\XX)$ on a hereditary  curve $\XX$ is essentially characterized by the properties listed in Theorem \ref{T:Characteristic} above; see \cite[Theorem IV.5.2]{ReitenvandenBergh} for the case of an algebraically closed field $\kk$ and \cite{LenzingReiten, Kussin} for further elaborations in the case of an arbitrary $\kk$.
\end{remark}

\begin{definition} Let $X$ be a complete regular curve over $\kk$. We say that $X_{\circ} \stackrel{\rho}\lar \NN$ is a \emph{weight function} if $\rho(x) = 1$ for all but finitely many points $x \in X_\circ$.  
\end{definition}

\begin{theorem}\label{T:Weighting} Let $X$ be a complete regular  curve over $\kk$,  $\eta \in \mathsf{Br}(\KK)$ be any Brauer class and $X_{\circ} \stackrel{\rho}\lar \NN$ be any weight function. 
Consider a homogeneous curve $\XX = (X, \kR)$  defined by $\eta$ (see Theorem \ref{T:NCRegular}). 
 Then there exists a  hereditary curve 
$\EE = \EE(X, \eta, \rho) =  (X, \kH)$ having the following properties.
\begin{enumerate}
\item[(a)] For any $x \in X_{\circ}$, the order $\widehat{\kH}_{x}$ is Morita equivalent to the order $H_{\rho(x)}(\kR_x)$. 
\item[(b)] We have: $[\FF_\XX] = \eta$.
\end{enumerate}
Let $(X', \eta', \rho')$ be another  datum as above and $\EE'$ be a  hereditary curve attached to it. Then the categories 
$\Coh(\EE)$ and $\Coh(\EE')$ are equivalent if and only if 
there exists an isomorphism 
$X \stackrel{f}\lar X'$ such that $f^\ast(\eta') = \eta \in \mathsf{Br}(\KK)$ and $\rho' f = \rho$.
\end{theorem}

\smallskip
\noindent
\emph{Proof} can be found in  \cite[Proposition 2.9]{Spiess}; see also \cite[Corollary 7.9]{BurbanDrozdMorita}. \qed

\begin{definition} A complete non-commutative  curve $\XX$ over a field $\kk$ is called \emph{exceptional} if its bounded derived category of coherent sheaves $D^b\bigl(\Coh(\XX)\bigr)$  admits a tilting object. Equivalently, there exists a finite-dimensional $\kk$-algebra $T$ and an exact equivalence of triangulated categories 
$
D^b\bigl(\Coh(\XX)\bigr) \lar D^b(T\mathrm{-}\mathsf{mod}).
$
\end{definition}
\begin{remark} The concept of an exceptional hereditary non-commutative curve was introduced for the first time  by Lenzing in \cite[Section 2.5]{Lenzing}, following an axiomatic characterization of such categories.  At this place let us  mention that there are various classes  of exceptional non-commutative curves which are not hereditary; see for instance \cite{BurbanDrozdAuslanderC, bdnpdalcurves, BDG}.
\end{remark}
\begin{theorem}\label{T:Lenzing} Let $\XX = (X, \kR)$ be an exceptional homogeneous  curve. Then there exists a tilting object $\kF \in \VB(\XX)$ such that 
\begin{equation}\label{E:TameBimodule}
\Lambda := \bigl(\End_{\XX}(\kF)\bigr)^\circ \cong
\left(
\begin{array}{cc}
\ff & \mathbbm{w} \\
0   & \gg \\
\end{array}
\right),
\end{equation}
where $\mathbbm{f}$ and $\mathbbm{g}$ are finite dimensional division algebras over $\kk$ and $\mathbbm{w}$ is a \emph{tame} 
$(\ff$--$\gg)$-bimodule (this means that $\dim_{\ff}(\mathbbm{w}) \cdot \dim_{\gg}(\mathbbm{w}) = 4$; see  \cite{DlabRingel}). Moreover, $g(X) = 0$.  
\end{theorem}

\noindent
\emph{Comment to the proof}. The first part of this theorem is due to Lenzing  \cite[Theorem 4.5]{Lenzing}. The statement $g(X) = 0$ can be deduced from results of  \cite[Section 4.1]{ArtindeJong}; see also \cite{KussinMemoirs}.  \qed

\smallskip
Let $(X, \eta, \rho)$ be a datum as in  Theorem \ref{T:Weighting} with $g(X) = 0$ and $\eta \in  \mathsf{Br}(\KK)$ be \emph{exceptional}. The latter condition means that homogeneous  curve $\XX = (X, \kR)$ determined  by $\eta$ is exceptional. Let $\kF \in \VB(\XX)$ be a tilting object from Theorem \ref{T:Lenzing} and $T$ be the corresponding tilted algebra (\ref{E:TameBimodule}). Then we have an exact equivalence 
$$
\sT := \mathsf{RHom}_{\XX}(\kF, \,-\,): D^b\bigl(\Coh(\XX)\bigr) \lar D^b(\Lambda\mathrm{-}\mathsf{mod}). 
$$
Let $\mathfrak{E}_{\rho} := \bigl\{x \in X_\circ \, \big| \, \rho(x) \ge 2 \bigr\} = \bigl\{x_1, \dots, x_t\bigr\}$ be the \emph{special locus} of $\rho$. 
For any $1 \le i \le t$, let $\kS_i$ be the unique (up to isomorphisms) simple object of the category $\Tor_{x_i}(\XX)$ and $U_i := \Hom_{\XX}(\kF, \kS_i) \in \Lambda\mathrm{-}\mathsf{mod}$ be the corresponding regular left $\Lambda$-module. Of course, we have: $\sT\bigl(\kS_i[0]\bigr) \cong U_i[0]$, where 
$$
\Lambda\mathrm{-}\mathsf{mod} \lar D^b(\Lambda\mathrm{-}\mathsf{mod}), \; M\mapsto M[0] = \bigl( \dots \lar 0 \lar M \lar 0 \lar \dots \bigr)
$$
is the standard embedding. For any $1 \le i \le t$, let $A_i := \widehat{\kR}_{x_i}$ and $D_i = A_i/\mathsf{rad}(A_i)$. Then  
$$
D_i^\circ \cong \End_{\XX}(\kS_i) \cong \End_{\Lambda}(U_i).
$$
Recall that the duality functor
$
\Hom_{\kk}(-\,,\kk): \, \Lambda\mathrm{-}\mathsf{mod} \lar \mathsf{mod}\mathrm{-}\Lambda
$
is a contravariant equivalence of categories.  For any $1 \le i \le t$, consider  a $(D_i$--$\Lambda)$-bimodule  $V_i := \Hom_{\kk}(U_i, \kk)$. 
In this notation, we put:
\begin{equation}\label{E:Squid}
\Pi := 
\left[
\begin{array}{ccc|c|ccc|c}
D_1 & \dots & D_1 & & 0 & \dots & 0 & V_1 \\
\vdots & \ddots & \vdots & & \vdots & \ddots & \vdots & \vdots \\
0 & \dots & D_1 & & 0 & \dots & 0 & V_1 \\
\hline
\vdots & & \vdots & \ddots & \vdots & & \vdots & \vdots \\
\hline
0 & \dots & 0 & & D_t & \dots & D_t & V_t \\
\vdots & \ddots & \vdots & & \vdots & \ddots & \vdots & \vdots \\
0 & \dots & 0 & & 0 & \dots & D_t & V_t \\
\hline
0 & \dots & 0 & & 0 & \dots & 0 & \Lambda \\
\end{array}
\right],
\end{equation}
where each $D_i$ occurs precisely $m_i:= \rho(x_i)-1$ times on the diagonal. 

\begin{theorem}\label{T:TiltingMain} Let $X$ be a complete regular curve over $\kk$ of genus zero, $\eta \in \mathsf{Br}(\KK)$ be an exceptional class,  $X_{\circ} \stackrel{\rho}\lar \NN$ a {weight function} and $\EE = \EE(X, \eta, \rho) = (X, \kH)$ be a  hereditary  curve attached to this datum (see Theorem \ref{T:Weighting}). Then there exists an exact equivalence 
\begin{equation}\label{E:TiltingEq}
D^b\bigl(\Coh(\EE)\bigr) \simeq D^b(\Pi\mathrm{-}\mathsf{mod}),
\end{equation} 
where $\Pi$ is the $\kk$-algebra given by (\ref{E:Squid}). In other words, the curve $\EE$ is exceptional. 
\end{theorem}

\begin{proof} Consider a homogeneous  curve  $\XX = (X, \kR)$  determined by $\eta \in \mathsf{Br}(\KK)$. Without loss of generality one may  assume that $\FF := \Gamma(X, \kK \otimes_\kO \kR)$ is a skew field. Then  there exists $m \in \NN$ such that $\Gamma(X, \kK \otimes_\kO \kH) \cong M_m(\FF).
$

\smallskip
\noindent
For any $1 \le i \le t$ we have an isomorphism of $\kk$-algebras 
$
H_i := \widehat{\kH}_{x_i} \cong H(A_i, \vec{p}_i),
$
where $A_i = \widehat{\kR}_{x_i}$ and $\vec{p}_i \in \NN^{\rho(x_i)}$ is some vector. In particular, 
there are 
precisely $\rho(x_i) = m_i+1$ pairwise non-isomorphic simple left $H_i$--modules $S_i^{(0)}, S_i^{(1)}, \dots, S_i^{(m_i)}$ with a cyclic ordering  such that 
\begin{equation}\label{E:ActionTau}
\tau(S_i^{(j)}) \cong S_i^{(j+1)} \; \mbox{\rm for all} \; 1 \le i \le t \; \mbox{\rm and} \; 0 \le j \le m_i. 
\end{equation}
Let $P_i^{(j)}$ be an indecomposable projective left $H_i$--module defined by (\ref{E:IndecompProjectives}) such that $$\Hom_{H_i}\bigl(P_i^{(j)}, S_i^{(j)}\bigr) \ne 0.$$
According to \cite[Theorem 6.2]{BurbanDrozdMorita} there exists $\kP \in \Pic(\EE)$ such that 
$\widehat{\kP}_{x_i} \cong P_i^{(0)}$ for all $1 \le i \le t$. Let $\kA := \bigl(\mathit{End}_{\XX}(\kP)\bigr)^\circ$. It is clear that 
$\widehat{\kA}_x \cong \widehat{\kR}_x$ for all $x \in X$ and 
$\Gamma(X, \kK \otimes_\kO \kA) \cong \FF$. It follows that $\YY := (X, \kA)$ is a complete homogeneous  curve over $\kk$ and by Theorem \ref{T:NCRegular} we have: $\Coh(\YY) \simeq \Coh(\XX)$. In particular, the curve $\YY$ is exceptional. 

Following the terminology of \cite[Definition 4.1]{Minors}, the homogeneous  curve $\YY$ is a \emph{minor} of the hereditary curve $\EE$. 
We have the following functors:
\begin{itemize}
\item $\sG := \sHom_\kH(\kP, \,-\,)$ from  $\Coh(\EE)$ to $\Coh(\YY)$.
\item $\sF := \kP \otimes_\kA \,-\,$  from $\Coh(\YY)$ to $\Coh(\EE)$.
\end{itemize}
Note that $(\sF, \sG)$ is an  adjoint pair and both functors $\sF$ and $\sG$ are exact. The general theory of minors developed in 
\cite[Section 4]{Minors} leads to the following results. 

First note that $\sF$ is fully faithful. Next,  denote by 
$\sD\sG$ and  $\sD\sF$  the corresponding derived functors between the bounded derived categories of coherent sheaves $D^b\bigl(\Coh(\EE)\bigr)$ and $D^b\bigl(\Coh(\YY)\bigr)$. Then  $(\sD\sF, \sD\sG)$ is again an adjoint pair and $\sD\sF$ is fully faithful.
 
Consider the  sheaf $\kI = \kI_{\kP}$ of two-sided ideals in $\kH$ defined as follows: 
$$
\kI := \mathit{Im}\bigl(\kP \otimes_{\kA} \kP^\vee 
\stackrel{\mathit{ev}}\lar\kH\bigr),
$$
where $\mathit{ev}$ is the evaluation morphism. 
It is clear that  $\kI_x = \kH_x$ for all $x \in X_\circ \setminus \mathfrak{E}_{\rho}$ and 
$
\overline\kH := \kH/\kI 
$ is supported at $\mathfrak{E}_{\rho}$. 
 One can check that  for any $1 \le i \le t$ we have: 
\begin{equation*}
\widehat{\kI}_{x_i} = 
\left[
\begin{array}{cccc}
A_i & J_i & \dots & J_i \\
A_i & J_i & \dots & J_i\\
\vdots & \vdots & \ddots & \vdots \\A_i  & J_i & \dots & J_i\\
\end{array}
\right]^{\underline{(p_0^{(i)}, \dots, p_{m_i}^{(i)})}}
\end{equation*}
where $\bigl(p_0^{(i)}, \dots, p_{m_i}^{(i)}\bigr) = \vec{p}_i$.  Let $L := \Gamma\bigl(X, \overline\kH\bigr)$. Then we have:
$
L \cong L_1 \times \dots \times L_t,
$
where 
$$
L_i \cong \overline{\kH}_{x_i} \cong 
\left[
\begin{array}{cccc}
D_i & 0 & \dots & 0 \\
D_i & D_i & \dots & 0\\
\vdots & \vdots & \ddots & \vdots \\
D_i  & D_i & \dots & D_i\\
\end{array}
\right]^{\underline{(p_1^{(i)}, \dots, p_{m_i}^{(i)})}} 
$$
for all $1 \le i \le t$. It is clear, that $L_i$ is Morita equivalent to the algebra
$$
\left[
\begin{array}{cccc}
D_i & 0 & \dots & 0 \\
D_i & D_i & \dots & 0\\
\vdots & \vdots & \ddots & \vdots \\
D_i  & D_i & \dots & D_i\\
\end{array}
\right] \subset M_{m_i}(D_i).
$$
For any $\kE^\bu \in D^b\bigl(\Coh(\EE)\bigr)$ we have a distinguished triangle
$$
(\sD\sF \circ\sD\sG)(\kE^\bu) \stackrel{\xi_{\kE^\bu}}\lar \kE^\bu \lar \kC^\bu_{} \lar
 (\sD\sF \circ \sD\sG)(\kE^\bu)[1],
$$
where $\sD\sF \circ\sD\sG \stackrel{\xi}\lar \mathsf{Id}$ is the counit of the adjoint pair $(\sD\sF, \sD\sG)$. Since $\sD\sF$ is fully faithful, 
the morphism $\sD\sG(\xi_{\kE^\bu})$ is an isomorphism and, as a consequence, 
$\sD\sG(\kC^\bu) = 0$. The kernel $\mathsf{Ker}(\sD\sG)$ of the functor $\sD\sG$ consists of those complexes, whose cohomology is annihilated by the sheaf of ideals $\kI$. 
Note that for any $1 \le i \le t$ the ideal $\widehat{\kI}_{x_i}$ is projective (hence, flat), viewed as a right $\widehat{\kH}_{x_i}$-module. It implies that $\mathsf{Ker}(\sD\sG)$
can be identified
with the derived category $D^b(L\mathsf{-mod})$; see  \cite[Theorem 4.6]{Minors}. Let $D^b\bigl(L\mathsf{-mod})\bigr) \stackrel{\sI}\lar 
D^b\bigl(\Coh(\EE)\bigr)$ be the corresponding fully faithful embedding, whose essential image is $\mathsf{Ker}(\sD\sG)$.
Then we  get a semi-orthogonal decomposition 
\begin{equation}\label{E:SemiOrth}
D^b\bigl(\Coh(\EE)\bigr) = \bigl\langle \mathsf{Im}(\sI), \, \mathsf{Im}(\sD\sF)\bigr\rangle = \left\langle
D^b\bigl(L\mathsf{-mod}), \,
D^b\bigl(\Coh(\YY)\bigr)
\right\rangle,
\end{equation}
see \cite[Theorem 4.5]{Minors}.
For any $1 \le i \le t$ and $1 \le j \le m_i$ consider the following $L_i$-modules
$Z_i^{(j)}$ given in terms of their projective resolutions
$$
\left\{
\begin{array}{l}
0 \lar P_i^{(0)} \lar P_i^{(1)} \lar Z_i^{(1)} \lar 0 \\
0 \lar P_i^{(m_i)} \lar P_i^{(1)} \lar Z_i^{(2)} \lar 0 \\
\vdots \\
0 \lar P_i^{(2)} \lar P_i^{(1)} \lar Z_i^{(m_i)} \lar 0. \\
\end{array}
\right.
$$
Note that $Z_i := \bigoplus\limits_{j = 1}^{m_i} Z_i^{(j)}$ is an injective cogenerator of the category $L_i\mathsf{-mod}$. Let
$
Z := \bigoplus\limits_{i = 1}^{t} Z_i
$
and $\kZ[0] := \sI(Z)$, then we have:  $\kZ \in \Tor(\XX)$. Next, we set $\widetilde\kF := \sF(\kF) \in \VB(\EE)$, where $\kF \in \VB(\YY)$ is a tilting object from Theorem \ref{T:Lenzing}. We claim that
\begin{equation}
\kX^\bu:= \kZ[-1] \oplus \widetilde{\kF}[0]
\end{equation}
is a tilting object in  the derived category $D^b\bigl(\Coh(\EE)\bigr)$.

The statement  that $\kX^\bu$ generates $D^b\bigl(\Coh(\EE)\bigr)$ follows from existence of  a semi-orthogonal decomposition 
(\ref{E:SemiOrth})
and the facts that $Z$ generates $D^b\bigl(L\mathsf{-mod}\bigr)$ and
$\kF$ generates $D^b\bigl(\Coh(\YY)\bigr)$. Since both functors $\sI$ and $\sD\sF$ are fully faithful and $Z$ and $\kF$ are tilting objects in the  corresponding derived categories, we have:
$$
\Ext^i_{\EE}(\kZ, \kZ) = 0 = \Ext^i_{\EE}(\widetilde{\kF}, \widetilde{\kF})
$$
for all $i \ge 1$. Since the functor $\sD\sF$ is left adjoint to $\sD\sG$ and $\sD\sG(\kZ) =  0$, we have:
$$
\Ext^i_{\EE}(\widetilde{\kF}, \kZ) \cong  \Hom_{D^b(\EE)}\bigl(\sD\sF(\kF), \kZ[i]\bigr) \cong \Hom_{D^b(\YY)}\bigl(\kF, \sD\sG(\kZ)[i]\bigr) = 0 \quad \mbox{\rm for all} \; i \in \ZZ.
$$
This vanishing is also a consequence of the semi-orthogonal decomposition (\ref{E:SemiOrth}).
Finally, for any $i \in \ZZ$ we have:
$
\Ext^i_{\EE}(\kZ, \widetilde{\kF}) \cong \Gamma\bigl(X, \sExt^i_{\kH}(\kZ, \widetilde{\kF})\bigr).
$
Since $\kZ$ is torsion and $\widetilde{\kF}$ is locally projective, we have: $\sHom_{\kH}(\kZ, \widetilde{\kF}) = 0$. As  $\EE$ is hereditary, we also have:  $\sExt^{i}_{\kH}(\kZ, \widetilde{\kF}) = 0$ for all $i \ge 2$. Therefore, $\Hom_{D^b(\EE)}\bigl(\kX^{\bu}, \kX^{\bu}[i]\bigr) = 0$
for $i \ne 0$. We have shown that $\kX^{\bu}$ is a tilting object in 
$D^b\bigl(\Coh(\EE)\bigr)$.  Put  
\begin{equation}\label{E:Tilted2}
\Pi:=
 \bigl(\End_{D^b(\EE)}(\kX^{\bu})\bigr)^\circ
\cong
\left( 
 \begin{array}{cc}
 \bigl(\End_{\EE}(\kZ)\bigr)^\circ & \Ext^1_{\EE}(\kZ, \widetilde\kF) \\
 0 & \bigl(\End_{\EE}(\widetilde\kF)\bigr)^\circ
 \end{array}
\right).
\end{equation} Then  the triangulated categories $D^b\bigl(\Coh(\EE)\bigr)$ and $D^b(\Pi\mathsf{-mod})$ are equivalent; see \cite{Keller}.
 
\smallskip
\noindent
Note that $
\bigl(\End_{\EE}(\widetilde{\kF})\bigr)^\circ \cong \bigl(\End_{\YY}(\kF)\bigr)^\circ = \Lambda
$ and 
$
\End_{\EE}(\kZ) \cong \End_{L}(Z) \cong \prod\limits_{i = 1}^t \End_{L_i}(Z_i).$
An easy computation shows that
$$
\bigl(\End_{L_i}(Z_i)\bigr)^\circ \cong 
\left[
\begin{array}{cccc}
D_i & D_i & \dots & D_i \\
0 & D_i & \dots & D_i\\
\vdots & \vdots & \ddots & \vdots \\
0  & 0 & \dots & D_i\\
\end{array}
\right] \subset M_{m_i}(D_i).
$$
Finally, using the Auslander--Reiten duality formula (\ref{E:ARDuality}) and the fact that $(\sF, \sG)$ is an adjoint pair, we get binatural isomorphisms
$$
\Ext^1_{\EE}\bigl(\kZ, \widetilde\kF\bigr) \cong \Hom_{\EE}\bigl(\sF(\kF), \tau^{-1}(\kZ)\bigr)^\ast \cong
 \Hom_{\YY}\bigl(\kF, 
\sG\bigl(\tau^{-1}(\kZ)\bigr)\bigr)^\ast.
$$ 
 Next, we have:
 $
 \sG\bigl(\tau^{-1}(\kZ)\bigr) \cong \bigoplus\limits_{i = 1}^t \kS_{i}^{\oplus m_i}
 $
 where $\kS_i$ is the unique (up to isomorphism) simple object of the category $\Tor_{x_i}(\YY)$. Hence, we get  isomorphisms 
 $$\Hom_{\YY}\bigl(\kF, 
\sG\bigl(\tau^{-1}(\kZ)\bigr)\bigr) \cong 
\bigoplus\limits_{i = 1}^t \Hom_{\YY}\bigl(\kF, 
\kS_i\bigr)^{\oplus m_i} \cong 
\bigoplus\limits_{i=1}^t U_i^{\oplus m_i}.
$$
Taking the duals over $\kk$, we  get a bimodule  isomorphism
$
\Ext^1_{\EE}\bigl(\kZ, \widetilde\kF\bigr) \cong \bigoplus\limits_{i=1}^t V_i^{\oplus m_i}.
$
This implies that  the $\kk$-algebras  given by (\ref{E:Squid}) and (\ref{E:Tilted2}) are isomorphic.
\end{proof}

\begin{remark} Let $\kk$ be an algebraically closed field and  $X = \PP^1_{\kk}$. We chose homogeneous coordinates $(u: v)$ on $X$. Then
$\kF := \kO(-1) \oplus \kO \in \VB(X)$ is a tilting bundle  $u$ and $v$ can be viewed as elements of a distinguished basis of $\Hom_{X}\bigl(\kO(-1), \kO\bigr)$. Hence, $\Lambda := \bigl(\End_{X}(\kF)\bigr)^\circ$ can be identified with the path algebra of the Kronecker quiver $
\xymatrix{
\bu  \ar@/^/[r]^{u} \ar@/_/[r]_{v}  & \bu
}
$
and we have an exact equivalence 
$
\sT := \mathsf{RHom}(\kF, \,-\,): \;D^b\bigl(\Coh(X)\bigr) \lar D^b(\Lambda\mathrm{-}\mathsf{mod}).
$

Let $X_\circ \stackrel{\rho}\lar \NN$ be any weight function and 
$\mathfrak{E}_\rho = \left\{x_1, \dots, x_t\right\}$ be the corresponding special locus. We write $x_i = (\alpha_i : \beta_i)$ for all $1 \le i \le t$. Let $\kS_i \in \Tor_{x_i}(X)$ be the simple object and $U_i \in \Lambda\mathsf{-mod}$ be its image under the equivalence $\sT$ (i.e.~$\sT\bigl(\kS_i[0]\bigr) \cong U_i[0])$. Then $U_i = 
\xymatrix{
\kk  \ar@/^/[r]^{\alpha_i} \ar@/_/[r]_{\beta_i}  & \kk
}
$
and $\End_{\Lambda}(U_i) \cong \kk$ for all $1 \le i \le t$. Let $\rho(x_i) = m_i +1$. Then the algebra $\Pi$ defined by (\ref{E:Squid}) is isomorphic to the path algebra of the following quiver
\begin{equation}\label{E:TrueSquid}
\begin{array}{c}
\xymatrix{
                                    &     & \bu  \ar[r] & \bu \dots \bu \ar[r]^-{c_1^{(m_1)}} & \bu  \\
\bu  \ar@/^/[r]^{u} \ar@/_/[r]_{v}  & \bu \ar[r]^-{c_{i}^{(1)}} 
\ar[ru]^-{c_{1}^{(1)}} \ar[rd]_-{c_{t}^{(1)}} & \bu  \ar[r] & \bu \dots   \bu \ar[r]^-{c_i^{(m_i)}} & \bu \\
                                    &     & \bu \ar[r] & \bu  \dots  \bu \ar[r]^-{c_t^{(m_t)}} & \bu }
                                    \end{array}
\end{equation}
subject to the relations $c_i^{(1)}(\beta_i u - \alpha_i v) = 0$ for all $1 \le i \le t$. This is a so-called \emph{squid algebra} (see \cite[Section 
IV.6]{BrennerButler} and \cite[Section 4]{Ringel}). The \emph{canonical algebra} $\Sigma$ 
attached to  the same datum  $\bigl((x_1, m_1+1), \dots, (x_t, m_{t}+1)\bigr)$ is the path algebra of the quiver
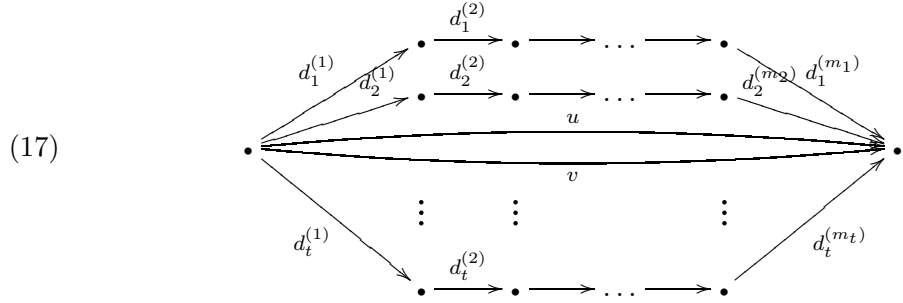
\def\bvd{\mbox{\huge$\boldsymbol{\vdots}$}}
\begin{equation}\label{E:Canonical}
\begin{array}{c}
 \xymatrix@R=1em{ 
 && \bu \ar[r]^{d_{1}^{(2)}} & \bu \ar[r] &\dots \ar[r] & \bu \ar[rrdd]^{d_{1}^{(m_1)}} 
 \\
 				  &&  \bu \ar[r]^{d_{2}^{(2)}} & \bu \ar[r] &\dots \ar[r] & \bu \ar[rrd]^(.2){d_{2}^{(m_2)}} \\
 				  \bu  \ar[uurr]^{d_{1}^{(1)}} \ar[urr]^(.8){d_{2}^{(1)}} 
 				  \ar[ddrr]_{d_{t}^{(1)}} \ar@/^/[rrrrrrr]^u \ar@/_/[rrrrrrr]_v &&&&&&& \bu \\
 					&& \bvd & \bvd & & \bvd \\
 				  && \bu  \ar[r]^{d_{t}^{(2)}} & \bu\ar[r] &\dots \ar[r] & \bu  \ar[rruu]_{d_{t}^{(m_t)}}
 				  }
 				  \end{array}
\end{equation}
modulo the relations 
\begin{equation}\label{E:CanonicalRel}
d_{i}^{(m_i)} \dots d_{i}^{(1)}=\beta_{i} u - \alpha_{i} v \quad  \mbox{\rm for}\quad 
 1 \le i \le t,
\end{equation} 
see \cite{RingelTameAlgebras}.  Then there exists an exact equivalence of triangulated categories
\begin{equation}\label{E:EquivalenceCanonicalSquid}
D^b\bigl(\Pi\mathsf{-mod}\bigr) \simeq D^b\bigl(\Sigma\mathsf{-mod}\bigr),
\end{equation}
see \cite{RingelTameAlgebras, Ringel}. For $t \ge 3$ one may without loss of generality assume that  $x_1 = (1:0)$, $x_2 = (0: 1)$ and $x_3 = (1:1)$. Suppose  that $t=3$. If $l_i = \rho(x_i)$ then we use the notation $\Pi_{(l_1, l_2, l_3)}$ and $\Sigma_{(l_1, l_2, l_3)}$ for the corresponding squid and canonical algebras, respectively. 

In the case of an arbitrary field $\kk$,  the algebra $\Pi$ given by (\ref{E:Squid}) is a variation of a squid algebra  introduced by Ringel in  \cite[Section 4]{Ringel}.
\end{remark}

\begin{remark} Let us mention that Theorem \ref{T:TiltingMain} is not entirely  original; see e.g.~\cite[Theorem 2.8 and Theorem 3.4]{HappelReiten} as well as  \cite{LenzingdelaPena, KussinMemoirs}. However, that works are based on the  ``axiomatic  approach'' to non-commutative hereditary curves and analogues of the derived equivalence (\ref{E:TiltingEq}) serve rather as a definition of $\EE$ than as its property.
\end{remark}

\section{Generalities on skew group products}
Let $A$ be a ring, $G$ be a finite group and $G \stackrel{\phi}\lar \Aut(A)$ be a group homomorphism. For any $g \in G$, let $A \stackrel{\phi_g}\lar A$ be the corresponding ring automorphism of $A$. The associated  skew group ring  $A\bigl[G, \phi\bigr]$ is a free left $A$--module of rank $|G|$ 
\begin{equation}
A\bigl[G, \phi\bigr] = \Bigl\{\sum\limits_{g \in G} a_g[g]  \, \big| a_g \in A  \Bigr\}
\end{equation}
equipped with the product given by the rule
$$
a[f] \cdot b[g] := a \phi_f(b)  [fg] \; \mbox{\rm for any} \; a, b \in  
A \; \mbox{\rm and} \;  f, g \in G.
$$
Then  $A\bigl[G, \phi\bigr]$ is a unital ring, whose multiplicative unit element is $1[e]$, where $1$ is the unit in $A$ and $e$ is the neutral element of $G$.  Let
$$
A^G := \left\{a \in A \, \big| \, \phi_g(a) = a \; \mbox{\rm for all} \, g \in G  \right\}
$$
be the ring of invariants. If $A$ is commutative then $A^G$ is the center of 
$A\bigl[G, \phi\bigr]$.  In what follows, we put $n = |G|$. 

\begin{lemma}\label{L:SkewGroupOfFields} Let $L$ be a field, $G \stackrel{\phi}\lar \Aut(A)$ be injective and $K = L^G$. Then we have an isomorphism of $K$-algebras 
\begin{equation}
L\bigl[G, \phi\bigr] \cong M_n(K).
\end{equation}
\end{lemma}

\begin{proof} By Artin's Theorem (see e.g.~\cite[Theorem VI.1.8]{Lang}) $L/K$ is a finite Galois extension and $G \cong \mathsf{Gal}(L/K)$.  Next, we have a group isomorphism 
\begin{equation}\label{E:RelativeBrauer}
H^2(G, L^\ast) \stackrel{\cong}\lar \mathsf{Br}(L/K), [\omega] \mapsto L[G, (\phi, \omega)]
\end{equation}
see e.g.~\cite[Theorem 5.6.6]{DrozdKirichenko}. Here, $L[G, (\phi, \omega)]$
is the crossed product of $L$ and $G$ with respect to the two-cocycle 
$G \times G \stackrel{\omega}\lar L^\ast$; see \cite{ReitenRiedtmann}. If $\omega$ is the trivial cocycle then $L[G, (\phi, \omega)] = L[G, \phi]$. Hence, we have an isomorphism of $K$-algebras $L[G, \phi] \cong M_m(K)$ for some $m \in \NN$. From the dimension reasons it follows that $m = n$.
\end{proof}

\begin{lemma}\label{L:ReductionSkewGroupAction} Let $A = A_1 \times \dots \times A_t$, where $A_i$ is a connected ring for all $1 \le i \le t$. Let $e_i := (0, \dots, 0, 1, 0, \dots, 0)$ be the $i$-th central idempotent of $A$. Assume that $G$ acts transitively on the set $\bigl\{e_1, \dots, e_t\bigr\}$. Let $A_{\diamond} = A_1$, $G_\diamond$ be the stabilizer of $e_1$ and $G_\diamond \stackrel{\phi_\diamond}\lar \Aut(A_\diamond)$ be the restricted  action. 
Then the skew group rings  $A\bigl[G, \phi\bigr]$ and  $A_\diamond\bigl[G_\diamond, \phi_\diamond\bigr]$ are Morita equivalent. 
\end{lemma}

\begin{proof} By the transitivity assumption, for any $1 \le i, j \le t$ there exists $g \in G$ such that $\phi_g(e_i) = e_j$. Then we have: 
$e_j = [g] e_i [g]^{-1}$. Since $1_{A[G, \phi]} = e_1 + \dots + e_t$ and the 
idempotents $\bigl\{e_1, \dots, e_t\bigr\}$  are orthogonal and pairwise conjugate, the rings
$A[G, \phi]$ and $e_1 A[G, \phi] e_1$ are Morita equivalent. Now we prove that $e_1 A[G, \phi] e_1 \cong  A_\diamond\bigl[G_\diamond, \phi_\diamond\bigr]$.
Let $\bigl\{g_1, \dots, g_s\bigr\} \subset G$ be such that $g_1 = e$ and $G = g_1 G_\diamond \sqcup \dots  \sqcup g_s G_\diamond$. Consider an arbitrary element $A \ni a = (a_1, \dots, a_t) = a_1 + \dots + a_t$, where $a_i \in A_i$ for $1 \le i \le t$ as well as an arbitrary element $g \in G$. First note that $e_1 a = a_1$. Next, there exist unique $1 \le j \le s$ and $h \in G_\diamond$ such that $g = g_j h$. Then we have: $\phi_h(e_1) = e_1$ and 
$$
e_1 \cdot a[g] e_1 = a_1 [g_j h] e_1 = a_1 \phi_{g_j}(e_1) [g_j h] = 
\left\{
\begin{array}{cc}
a_1 [h] & \mbox{if}\;  j = 1 \\
0       & \mbox{otherwise}.
\end{array}
\right.
$$
Hence, $e_1 A[G, \phi] e_1 \cong  A_\diamond\bigl[G_\diamond, \phi_\diamond\bigr]$, as asserted. 
\end{proof}

\smallskip
\noindent
From now on, let $\kk$ be a field such that $\mathsf{gcd}\bigl(n, \mathsf{char}(\kk)\bigr) = 1$, $A$ be a $\kk$-algebra and $G \stackrel{\phi}\lar \Aut_{\kk}(A)$ be a group homomorphism. Then the skew product 
$A[G, \phi]$ is a $\kk$-algebra. 

\begin{theorem}\label{T:SkewGroupOrders} Let $A$ be a commutative connected Dedekind $\kk$-algebra, $O = A^G$ and $H = A[G, \phi]$. Then the following statements are true.
\begin{enumerate}
\item[(i)] $O$ is again a Dedekind $\kk$-algebra and $O \subseteq A$ is a finite extension. 
\item[(ii)] $H$ is a hereditary order, whose center is $O$ and whose rational hull is $M_n(K)$, where $K$ is the quotient field of $O$. 
\end{enumerate}
\end{theorem}

\begin{proof} For the first statement, see for instance \cite[Theorem 4.1]{BurbanDrozdSurface}. We conclude  that $O \subseteq H$ is finite and $H$ is a torsion free module over $O$. It follows from 
Lemma \ref{L:SkewGroupOfFields} that the rational hull of $H$ is $M_n(K)$. Hence $H$ is an order, whose center is $O$. Finally, it follows from \cite[Theorem 1.3]{ReitenRiedtmann} that $H$ is hereditary; see also \cite[Corollary 2.7]{BurbanDrozdQuotients}.
\end{proof}

\begin{lemma}\label{L:ActionDVR} Let $\kk$ be algebraically closed, $A = \kk\llbracket z\rrbracket$ and $G \stackrel{\phi}\lar \Aut_{\kk}(A)$ be an \emph{injective} group homomorphism.  Then the following statement are true:
\begin{enumerate}
\item[(i)] The group $G$ is cyclic, i.e.~$G \cong \ZZ_n$.
\item[(ii)] We have: $A[G, \phi] \cong H_n(O)$, where $O = \kk\llbracket z^n\rrbracket$. 
\end{enumerate}
\end{lemma}
\begin{proof} Let $\idm =(z)$ be the maximal ideal in $A$. For any $g \in G$, let $\idm/\idm^2 \stackrel{\bar{\phi}_g}\lar  \idm/\idm^2$ be the induced automorphism.
We identify $\kk$ with $\idm/\idm^2$ sending $1$ to  $[z]$. Then 
$\bar{\phi}_g([z]) = \xi_g [z]$ for some $\xi_g \in \kk^\ast$.  
Clearly, $\xi_e = 1$ and $\xi_{g_1 g_2} = \xi_{g_1} \xi_{g_2}$ for all $g_1, g_2 \in G$. 
Next, for any $g \in G$ consider the   automorphism of $\kk$-algebras 
$$
A \stackrel{\psi_g}\lar A, \; f(z) \mapsto f(\xi_g z).
$$
We define $\tau \in \Aut_{\kk}(A)$ by the rule
$
\tau(z) = \frac{1}{n} \sum\limits_{g \in G} \psi_g^{-1} \phi_g(z). 
$
It is easy to see that $\psi_e = \mathsf{id}$, $\psi_{g_1 g_2} = \psi_{g_1} \psi_{g_2}$ and $\psi_g \tau = \tau \phi_g$ for all $g_1, g_2, g \in G$. Hence, $\tau$ can be extended to an isomorphism of  $\kk$-algebras  
$A[G, \phi] \stackrel{\tau}\lar A[G, \psi]$.

Since $\phi$ is injective, $G \stackrel{\psi}\lar  \Aut_{\kk}(A)$ is injective, too. It follows that $G \lar \kk^\ast, g \mapsto \xi_g$ is an injective group homomorphism. Moreover,  $\xi_g^n = 1$ for all $g \in G$. This implies that $G$ is a cyclic group of order $n$. 

Let $h$ be a generator of $G$. Then $\xi = \xi_h$ is a primitive $n$-th root of $1$ in $\kk$.  For $ 1 \le k \le n$, let   $\zeta_k := \xi^k$ 
and 
\begin{equation}\label{E:PrimitiveIdempotents}
\varepsilon_k := \frac{1}{n} \sum\limits_{j=0}^{n-1} \zeta_k^j [h]^j \in A[G, \psi].
\end{equation}
Then we have:
$$
\left\{
\begin{array}{ccc}
1 & = & \varepsilon_1 + \dots + \varepsilon_n \\
\varepsilon_k \cdot \varepsilon_l & = & \delta_{kl} \varepsilon_k, \; 1\le k, l\le n. 
\end{array}
\right.
$$
In other words, $\bigl\{\varepsilon_1,\dots,\varepsilon_n\bigr\}$ is a complete set
of primitive idempotents of  $A[G, \psi]$.
An isomorphism $A[G, \psi] \stackrel{\mu}\lar \widehat{\kk\bigl[\vec{C}_n\bigr]}$ is given by the  rule:
\begin{equation}\label{E:IsomorphismCyclicQuiver}
\left\{
\begin{array}{ccc}
\varepsilon_k & \stackrel{\mu}\mapsto & e_k \\
\varepsilon_{k+1} z \varepsilon_{k}  & \stackrel{\mu}\mapsto & a_k,
\end{array}
\right.
\end{equation}
where $\widehat{\kk\bigl[\vec{C}_n\bigr]}$ is the complete path algebra of a cyclic quiver $\vec{C}_n$ (see (\ref{E:Cyclicquiver})) and $e_k \in \widehat{\kk\bigl[\vec{C}_n\bigr]}$ is the idempotent   corresponding to the vertex $1 \le k \le r$. This gives us the desired isomorphisms
$
A[G, \phi] \cong A[G, \psi] \cong  \widehat{\kk\bigl[\vec{C}_n\bigr]} \cong  H_n(O).
$
\end{proof}

\section{Equivariant coherent sheaves on regular curves and hereditary non-commutative curves}

As in the previous section, let $G$ be a finite group of order $n$ and $\kk$ be a field such that $\mathsf{gcd}\bigl(n, \mathsf{char}(\kk)\bigr) = 1$. Let $Y$ be a  quasi-projective variety  over $\kk$ and $G \stackrel{\gamma}\lar \Aut_{\kk}(Y)$ be a group homomorphism, which we assume to be injective. Then we have a quasi-projective  variety  $X := Y/G$ and a canonical projection $Y \stackrel{\pi}\lar X$. 
Let us now recall the corresponding constructions, following \cite{SGA1} (see also \cite[Appendix 1]{Mustata}). 

We can always find an open affine $G$-invariant covering $Y = Y_1 \cup \dots \cup Y_m$. For any $1 \le i \le m$ let $A_i = \kO_Y(Y_i)$. Then for any $g \in G$ we have a $\kk$-algebra automorphism $A_i \stackrel{\gamma^\sharp_{i, g}}\lar A_i$. Moreover, $\gamma^\sharp_{i, e} = \id$ and  
$\gamma^\sharp_{i, g_1 g_2} = \gamma^\sharp_{i, g_2} \gamma^\sharp_{i, g_1}$ for all $g_1, g_2 \in G$. For any $g \in G$ we put $\widehat\gamma^{(i)}_{g} := \gamma^\sharp_{i, g^{-1}}$. In this way, for any $1 \le i \le m$ we get a group homomorphism 
$G \stackrel{\widehat\gamma^{(i)}}\lar \Aut_{\kk}(A_i)$.  Let  $O_i := A_i^G$ and $X_i = \Spec(O_i)$. By the construction  of $X = Y/G$, we have an open affine covering $X = X_1 \cup \dots \cup X_m$ with $Y_i = \pi^{-1}(X_i)$. Moreover, the morphism $Y_i \stackrel{\pi_i}\lar X_i$ is dual to the inclusion $O_i \subseteq A_i$. Next, we put $H_i := A_i[G, \widehat{\gamma}^{(i)}]$. In this way we  construct a coherent sheaf of $\kO_X$-algebras  $\kH$ on  $X$ such that $\kH(X_i) = H_i$ for all $1 \le i \le m$.  

\begin{proposition} The following results are true.
\begin{enumerate}
\item[(a)] Assume $Y$ is integral. Then for any $1\le i \le m$, the homomorphism $G \stackrel{\widehat\gamma^{(i)}}\lar \Aut_{\kk}(A_i)$ is injective.
\item[(b)] Let $\kK$ be the sheaf of rational functions on $X$, $\KK = \Gamma(X, \kK)$ be the field of rational functions on $X$ and $\FF = \Gamma(X, \kK \otimes_\kO \kH)$. Then we have an isomorphism of $\KK$-algebras $\FF \cong M_n(\KK)$.

\item[(c)] Suppose furthermore that $Y$ is a regular curve. Then $X$ is regular as well and 
$\XX = Y \hspace{-1mm}\sslash \hspace{-1mm}G = (X, \kH)$ is a non-commutative hereditary curve. 
\item[(d)] Let $Y$ be as above,  $y \in Y_\circ$, $G_\diamond \subseteq G$ be its stabilizer group, $r = |G_\diamond|$, $x := \pi(y) \in X$, $O = \widehat{\kO}_x$ and $H = \widehat{\kH}_x$. If $\kk$ is algebraically closed then $H$ is Morita equivalent to the standard hereditary order $H_r(O)$. 
\end{enumerate}
\end{proposition}

\begin{proof} (a) Let $\LL$ be the field of rational functions on $Y$. Then for any $1 \le i \le m$ we have a commutative diagram 
$$
\xymatrix{
G \ar[r]^-{\widehat{\gamma}^{(i)}} \ar@{_{(}->}[d]_-{\gamma} & \Aut_{\kk}(A_i) 
\ar@{^{(}->}[d] \\
\Aut_\kk(Y) \ar@{^{(}->}[r] & \Aut_{\kk}(\LL)
}
$$
where three out four group homomorphisms are known to be injective. Hence, 
$\widehat\gamma^{(i)}$ is injective, too. 

\smallskip
\noindent
(b) Since $\KK = \LL^G$, this result   is a consequence of Lemma \ref{L:SkewGroupOfFields}.

\smallskip
\noindent
(c)  This statement follows from Theorem \ref{T:SkewGroupOrders}.

\smallskip
\noindent
(d)  Let $\pi^{-1}(x) = \bigl\{y_1, \dots ,y_t\bigr\}$ with $y = y_1$. For $1\le i \le t$ we put $B_i = \widehat{\kO}_{y_i}$ and $B:= B_1  \times \dots \times B_t$. Then we have an injective group homomorphism $G \stackrel{\widehat{\gamma}}\lar \Aut_{\kk}(B)$. Moreover, we have an isomorphism of $\kk$-algebras $H \cong B[G, \widehat{\gamma}]$. By Lemma \ref{L:ReductionSkewGroupAction}, the $\kk$-algebras $H$ and $B_\diamond[G_\diamond, \widehat{\gamma}_\diamond]$ are Morita equivalent, where $B_\diamond = B_1$ and  $G_\diamond \stackrel{\widehat{\gamma}_\diamond}\lar \Aut_{\kk}(B_\diamond)$ is the restricted action. Since $\gamma$ is injective,  $\widehat{\gamma}_\diamond$ is injective, too. If $\kk$ is algebraically closed, then by Lemma \ref{L:ActionDVR} we have:  $G_\diamond \cong \ZZ_r$ and $B_\diamond[G_\diamond, \widehat{\gamma}_\diamond] \cong H_r(O)$.
\end{proof}

For any $g \in G$ the  automorphism $Y \stackrel{\gamma_g}\lar Y$ induces a pair of $\kk$-linear auto-equivalences $\gamma_g^\ast$ and $ \gamma_{g \ast}: \Coh(Y) \lar \Coh(Y)$, which assign to a coherent sheaf on $Y$  its inverse (respectively, direct)  image sheaf. We have: $\gamma_{g_1 g_2 \ast} = \gamma_{g_1 \ast} \gamma_{g_2 \ast}$ and $\gamma_{g \ast} = \gamma_{g^{-1}}^\ast$ for all $g, g_1, g_2 \in G$. Hence, in what follows we shall assume that the canonical isomorphisms of functors 
$\gamma_{g_1 g_2}^\ast \stackrel{\cong}\lar \gamma_{g_2}^\ast \gamma_{g_1}^\ast$ are trivial for all $g_1, g_2 \in G$. 

\begin{definition} The category $\Coh^G(Y)$ of $G$-equivariant coherent sheaves on $Y$ is defined as follows. 
\begin{enumerate}
\item[(a)] Its objects  are tuples
$\bigl(\kF, (\alpha_g)_{g \in G}\bigr)$,  where 
$\kF \in \Coh(Y)$ and $\kF \stackrel{\alpha_g}\lar \gamma_g^\ast(\kF)$ is an isomorphism in $\Coh(Y)$ for any $g \in G$ such that $\alpha_e = \id$ and 
\begin{equation}
\alpha_{g_2 g_1} = \gamma_{g_1}^\ast(\alpha_{g_2})\alpha_{g_1} \in \Hom_{Y}\bigl(\kF, \gamma_{g_2 g_1}^\ast(\kF)\bigr)
\end{equation}
for any  $g_1, g_2 \in G$.
\item[(b)] A morphism $\bigl(\kF, (\alpha_g)_{g \in G}\bigr) \lar 
\bigl(\kF', (\alpha'_g)_{g \in G}\bigr)$ of $G$-equivariant coherent sheaves 
is given by a morphism $f \in \Hom_Y(\kF, \kF')$ such that the diagram 
\begin{equation}\label{E:Equiv1}
\begin{array}{c}
\xymatrix{
\kF \ar[r]^-{\alpha_g} \ar[d]_-{f} & \gamma_g^\ast(\kF) \ar[d]^-{\gamma_g^\ast(f)} \\
\kF' \ar[r]^-{\alpha'_g}  & \gamma_g^\ast(\kF')
}
\end{array}
\end{equation}
is commutative for all $g \in G$. 
\end{enumerate}
\end{definition}

\smallskip
\noindent
The following result is well-known to the experts. For the reader's convenience, we give below its proof. 

\begin{proposition}\label{P:EquivNonComm} The categories $\Coh^G(Y)$ and 
$\Coh(\XX)$ are equivalent. 
\end{proposition}

\begin{proof} We first prove the local statement. Let $A$ be a (commutative) $\kk$-algebra and $G \stackrel{\phi}\lar \Aut_{\kk}(A)$ be a group homomorphism.  Consider a left $A[G, \phi]$-module $M$. Then $M$ is also a left $A$-module and for any $g \in G$ we have a $\kk$-linear automorphism
$$
M \stackrel{\alpha_g}\lar M, x \mapsto [g]x. 
$$
We have: $\alpha_e = \id$ and $\alpha_{g_1} \alpha_{g_2} = \alpha_{g_1 g_2}$ for all $g_1, g_2 \in G$. Moreover,
$$
\alpha_g(ax) = [g]ax = \phi_g(a) [g] x = \phi_g(a) \alpha_g(x)
$$
for all $a \in A$ and $x \in M$. Conversely, let $M$ be  a left $A$-module and 
$\left(M \stackrel{\alpha_g}\lar M\right)_{g \in G}$ be a family of $\kk$-linear automorphisms such that $\alpha_g(ax) = \phi_g(a) \alpha_g(x)$ for any $a \in A$ and $x \in M$ and such that $\alpha_e = \id$ and $\alpha_{g_1} \alpha_{g_2} = \alpha_{g_1 g_2}$ for all $g_1, g_2 \in G$. Then  $M$ can be equipped with a unique structure of a left $A[G, \phi]$-module such that $[g] x = \alpha_g(x)$. 
In these terms, a morphism $\bigl(M, (\alpha_g)_{g \in G}\bigr) \stackrel{f}\lar 
\bigl(M', (\alpha'_g)_{g \in G}\bigr)$ of $A[G, \phi]$-modules is a morphism 
of $A$-modules $M \stackrel{f}\lar M'$ such that 
\begin{equation}\label{E:Equiv2}
\begin{array}{c}
\xymatrix{
M \ar[r]^-{\alpha_g} \ar[d]_-{f} & M \ar[d]^-{f} \\
M' \ar[r]^-{\alpha'_g}  & M'
}
\end{array}
\end{equation}
is commutative for all $g \in G$. 

Let $A'$ be another commutative $\kk$-algebra and $A \stackrel{\vartheta}\lar A'$ be a homomorphism of $\kk$-algebras. Let  $X' = \Spec(A') \stackrel{\nu}\lar X= \Spec(A)$ be the morphism of schemes induced by $\vartheta$. The functors of global sections give equivalences of categories 
$\Qcoh(Y) \simeq A\mathsf{-Mod}$ and $\Qcoh(Y') \simeq A'\mathsf{-Mod}$.
In this identification, for $M \in A\mathsf{-Mod}$ we have:
$\nu^\ast(M) = A' \otimes_A M$. For any $a \in A$ and $x \in M$ we have:
$\vartheta(a) \otimes x = 1 \otimes ax$. Now, consider a special case when $A' = A$. Then we have mutually inverse isomorphisms of $A$-modules
$M \rightarrow \nu^\ast(M), x \mapsto 1 \otimes x$ and $\nu^\ast(M) \rightarrow  M, a \otimes x \mapsto \vartheta^{-1}(a)x$.

Now, let $\kF \in \Coh(Y)$ and  $Y = Y_1 \cup \dots \cup Y_m$ be a $G$-invariant open affine covering. For any $1 \le i \le m$ let $A_i = \kO_Y(Y_i)$, $M_i = \kF(Y_i)$  and $H_i = A_i[G, \widehat{\gamma}^{(i)}]$.  Let $\bigl(\kF \stackrel{\alpha_g}\lar \gamma_g^\ast(\kF)\bigr)_{g \in G}$ be a family of isomorphisms in $\Coh(Y)$ making $\kF$ to an $G$-equivariant sheaf. For each $1 \le i \le m$
$\alpha_g^{(i)} = \alpha_g\big|_{Y_i}: M_i \lar M_i$
is a $\kk$-linear map satisfying the property 
$\alpha_g^{(i)}(a x) = \widehat{\gamma}^{(i)}(a) \alpha_g^{(i)}(x)$ for all
$a \in A_i$ and $x \in M_i$. The above discussion allows one to equip $M_i$ with a structure of a left $H_i$-module. Globalizing this correspondence, we equip $\kF$ with a structure of a left $\kH$-module. Comparing (\ref{E:Equiv1}) with  (\ref{E:Equiv2}) we conclude that we  get a functor $\Coh^G(Y) \stackrel{\sE}\lar \Coh(\XX)$. Moreover, the above discussion shows that $\sE$ is fully faithful and dense, hence an equivalence of categories. 
\end{proof}

\smallskip
\noindent
\textbf{Summary}. Let $Y$ be a complete regular curve over a field $\kk$ and  $G$ be a finite group of order  $n$ such that $\mathsf{gcd}\bigl(n, \mathsf{char}(\kk)\bigr) = 1$. Let $G \stackrel{\gamma}\lar \Aut_{\kk}(Y)$ be an injective group homomorphism, $X = Y/G$ and $\XX = Y \hspace{-0.5mm}\sslash  \hspace{-0.5mm}G = (X, \kH)$ be the corresponding non-commutative hereditary curve. Then $X$ is also complete  and the following statements are true. 
\begin{enumerate}
\item[(i)] Let $\KK$ be the field of rational functions of $\XX$. Then the class $[\FF_{\XX}]$ of $\XX$ in the Brauer group $\mathsf{Br}(\KK)$ is trivial, where $\FF_{\XX} = \Gamma(X, \kK \otimes_\kO \kH)$. 
\item[(ii)] Let $y \in Y_\circ$, $x = \pi(y) \in X$ and $G_y$ be the stabilizer of $y$. Then $\widehat{H}_x$ is Morita equivalent to 
$\widehat{\kO}_y[G_y, \widehat{\gamma}_y]$. 
\item[(iii)] If $\kk$ is algebraically closed then 
$\widehat{\kO}_y[G_y, \widehat{\gamma}_y] \cong H_r(\widehat{\kO}_x)$, where $r = \big|G_y\big|$. In particular, the special locus $\mathfrak{E}_{\XX}$ of the hereditary curve $\XX$ admits the following description. Let $y \in Y_\circ$ be such that $x = \pi(y)$. 
Then $x \in \mathfrak{E}_{\XX}$ if and only if $G_y \ne \{e\}$. Moreover, $\rho(x) = \big|G_y|$. 
\end{enumerate}

\begin{remark} In the case the field $\kk$ is algebraically closed of characteristic zero, the theory of non-commutative hereditary curves  was considered in  \cite{ChanIngalls} from the perspective  of algebraic stacks.
\end{remark}

\begin{theorem}\label{T:Equivariant} Let $\kk$ be a field of $\mathsf{char}(\kk) \ne 2$, $Y$ be a complete regular and geometrically integral  curve over $\kk$ and $G \subset \Aut_{\kk}(Y)$ be a finite group of order $n$ acting faithfully on $Y$. Assume that $\mathsf{gcd}\bigl(n, \mathsf{char}(\kk)\bigr) = 1$ and  $X = Y/G$ is a curve of genus zero. Then there exists a finite dimensional $\kk$-algebra $\Pi_{(Y, G)}$ such that we have an exact equivalence 
\begin{equation}\label{E:EquivalencesEquiv}
D^b\bigl(\Coh^G(Y)\bigr) \simeq D^b\bigl(\Pi_{(Y, G)}\mathsf{-mod}\bigr).
\end{equation}
\end{theorem}

\begin{proof} Any geometrically integral  regular projective curve $X$ over $\kk$ of genus zero is isomorphic to a plane conic 
\begin{equation}\label{E:Conics}
X_{(a, b)} := \mathsf{Proj}\bigl(\kk[x,y,z]/(a x^2 + b y^2 - z^2)\bigr) 
\end{equation}
 for some $a, b \in \kk^\ast$. Let 
$$
\Lambda_{(a, b)} = \left\langle i, j \, \big| \, i^2 = a, j^2 = b, ij = - ji\right\rangle_{\kk}
$$
be the corresponding generalized quaternion algebra. It was shown in \cite{KussinTilting} that there exists a tilting bundle $\kF \in \VB(X_{(a, b)})$ such that $\bigl(\End_X(\kF)\bigr)^\circ \cong \Lambda_{(a, b)}$. The statement is therefore a consequence of Theorem \ref{T:TiltingMain} and Proposition \ref{P:EquivNonComm}.
\end{proof}

\begin{example} Let $G \subset \SL_2(\CC)$ be a finite subgroup. Then $G$ acts on the complex projective line $Y = \PP^1$ by the fractional-linear transformations. Then $X = Y/G \cong \PP^1$. Let $\XX = Y \hspace{-0.5mm}\sslash \hspace{-0.5mm}G$ be the corresponding non-commutative hereditary curve.  Then there exists a finite-dimensional algebra $\Pi_{(\PP^1, G)}$ of the form (\ref{E:TrueSquid})  such that 
$$
D^b\bigl(\Coh^G(Y)\bigr) \simeq D^b\bigl(\Coh(\XX)\bigr) \simeq D^b\bigl(\Pi_{(\PP^1, G)}\mathsf{-mod}\bigr).
$$
Up to a conjugation, a classification of finite subgroups of $\SL_2(\CC)$ is well-known; see for instance \cite{Kostrikin}.  In all the cases, the cardinality of the exceptional set $\mathfrak{E}_{\XX}$ is either two or three.
The group $\Aut_{\CC}(\PP^1)$ acts transitively on the set of triples on distinct points of $\PP^1$.
In the case of two special  points, we may assume that $\mathfrak{E}_{\XX} = \bigl\{(0:1), (1:0)\bigr\}$. In the case of three special  points, we may assume that $\mathfrak{E}_{\XX} = \bigl\{(0:1), (1:0), (1:1)\bigr\}$.  Therefore, to define $\XX$, it is sufficient to specify the sequence $(a, b, c)$ of orders on non-trivial stabilizers of the $G$-action on $\PP^1$ (with $a \le b \le c$ and allowing $a = 1$ in the case there are only two special  points). The corresponding hereditary curve $\XX$ will be therefore denoted by $\PP^1_{(a, b, c)}$. The following cases can occur. 
\begin{enumerate}
\item[(a)] $G \cong \ZZ_n$ with $n \ge 2$. The corresponding weight sequence is $(n, n)$.
\item[(b)] $G \cong \mathbb{D}_n$ is a binary dihedral group with $n \ge 2$. The corresponding weight sequence is $(2, 2, n)$. 
\item[(c)] $G$ is a binary tetrahedral, octahedral or icosahedral group. The corresponding weight sequences are $(2, 3, 3)$, $(2, 3, 4)$ and $(2, 3, 5)$, respectively. 
\end{enumerate}
On the other hand, the simply-laced Dynkin diagrams are parametrized  by the  triples 
$(a, b, c) \in \NN^3$ such that $$
a \le b \le c \quad \; \mbox{and} \quad  \dfrac{1}{a} + \dfrac{1}{b}+ \dfrac{1}{c} > 1.
$$  Hence, we may write $\Pi_{(\PP^1, G)} = \Pi_{(a, b, c)}$. On the other hand, let $\Gamma_{(a, b, c)}$ be the path algebra of the corresponding Euclidean quiver. Then there exists an exact equivalence of triangulated categories 
$
D^b\bigl(\Pi_{(a, b, c)}\mathsf{-mod}\bigr) \simeq D^b\bigl(\Gamma_{(a, b, c)}\mathsf{-mod}\bigr),
$
see \cite[Section 4.3]{RingelTameAlgebras} and  \cite[Section XII.1]{SimSko}.
Hence, there exists an exact equivalence 
$$
D^b\bigl(\Coh(\PP^1_{(a, b, c)})\bigr) \simeq D^b\bigl(\Gamma_{(a, b, c)}\mathsf{-mod}\bigr).
$$
This striking observation was made for the first time by Lenzing in \cite{LenzingFirst}. Later it led to a development of the theory  of weighted projective lines of Geigle and Lenzing in \cite{GeigleLenzing}. An elaboration of the equivalence $
D^b\bigl(\Coh^G(\PP^1)\bigr)  \simeq D^b\bigl(\Pi_{(\PP^1, G)}\mathsf{-mod}\bigr)
$
in the framework  of genuine equivariant coherent sheaves on $\PP^1$ can be found in \cite{Kirillov, Orbifolds}.
\end{example}

\begin{example}\label{E:EllipticQuotients1}
Let $\kk$ be a field of  $\mathsf{char}(\kk) \ne 2$, $\lambda \in \kk^\ast\setminus \{1\}$ and 
$$
Y_\lambda = \mathsf{Proj}\bigl(\kk[x,y,z]/(zy^2 - x(x-z)(x-\lambda z))\bigr) 
$$
be an elliptic curve over $\kk$. Then $G = \left\langle \imath \,\big| \, \imath^2 = e\right\rangle \cong \ZZ_2$ acts on $Y_\lambda$ by the rule
$(x: y: z) \stackrel{\imath}\mapsto (x: -y: z)$. There are precisely four points of $Y_\lambda$  with non-trivial stabilizers: $(0:0:1)$, $(0:1:0)$, $(1:0:1)$ and $(\lambda: 0:1)$. Next, we have: $X = Y_\lambda/G \cong \PP^1_{\kk}$. Let $Y_\lambda \stackrel{\pi}\lar X$ be the canonical projection. One can choose homogeneous coordinates on $X$ so that the image of the set of four ramification points of $\pi$   is
$\mathfrak{E} = \bigl\{(0:1),  (1:0), (1:1), (\lambda:1)\bigr\}$. For any $x \in \mathfrak{E}$ we have $\rho(x) = 2$. Let
 $\Sigma_\lambda$ be the tubular canonical algebra of type $((2,2,2,2);\lambda)$ \cite{Ringel}, i.e.~the path algebra of the following quiver

\begin{equation}\label{E:tubular}
\begin{array}{c}
\xymatrix
{
        &           & \circ \ar[lld]_{a_1}
\ar[ld]^{a_2}  \ar[rd]_{a_3}  \ar[rrd]^{a_4}
      &         &       \\
\circ \ar[rrd]_{b_1}  & \circ \ar[rd]^{b_2}
&        & \circ \ar[ld]_{b_3}
& \circ \ar[lld]^{b_4}\\
        &           &\circ &         &       \\
}
\end{array}
\end{equation}
modulo the  relations $b_1 a_1 - b_2 a_2 = b_3 a_3$ and  $b_1 a_1 - \lambda b_2 a_2 = b_4 a_4$.
An  exact equivalence
 of triangulated  categories
 \begin{equation}\label{E:standardtilting}
 D^b\bigl(\Coh^G(Y_\lambda)\bigr) \lar D^b(\Sigma_\lambda\mathsf{-mod})
 \end{equation}
 was for the first time discovered by Geigle and Lenzing; see  \cite[Example 5.8]{GeigleLenzing}.
 The algebra $\Sigma_\lambda$ is derived-equivalent to the squid algebra (\ref{E:TrueSquid}) of the same type $((2,2,2,2);\lambda)$ (see \cite{RingelTameAlgebras, Ringel}), which is of course consistent with Theorem \ref{T:Equivariant}.
\end{example}

\begin{example}\label{E:EllipticQuotients2} Let $\kk = \CC$. Consider  the following finite group actions on the following complex elliptic curves. 

\noindent
(I) Let  $Y = \mathsf{Proj}\bigl(\kk[x,y,z]/(zy^2 - x^3 - z^3)\bigr)$ and $G = \left\langle \varrho \,\big|\, \varrho^6 = e\right\rangle \cong \ZZ_6$. Then $G$ acts on $Y$ by the rule $\varrho(x:y:z) = (\xi x: -y: z)$, where $\xi = \exp\left(\dfrac{2\pi i}{3}\right)$  and $Y/G \cong \PP^1$. Moreover, 
\begin{enumerate}
\item[(a)] The stabilizer of $(-1:0:1)$ is $\ZZ_2$.
\item[(b)] The stabilizer of $(0:1:1)$ is $\ZZ_3$.
\item[(c)] The stabilizer of $(0:1:0)$ is $\ZZ_6$.
\end{enumerate}
Combining the exact equivalences of triangulated categories (\ref{E:EquivalencesEquiv}) and (\ref{E:EquivalenceCanonicalSquid}) we get  
$$
D^b\bigl(\Coh^G(Y)\bigr) \simeq D^b\bigl(\Pi_{(2, 3,6)}\mathsf{-mod}\bigr)  \simeq D^b\bigl(\Sigma_{(2, 3,6)}\mathsf{-mod}\bigr),
$$
where $\Pi_{(2, 3,6)}$ and $\Sigma_{(2, 3,6)}$ are the squid and canonical algebras of type $(2, 3, 6)$, respectively.

\smallskip
\noindent
(II) Next, let $\widetilde\varrho = \varrho^4$. Consider the subgroup $\ZZ_3 \cong N = \langle \widetilde\varrho\rangle \subset G$. Then $N$ acts on $Y$ by the rule $\widetilde{\varrho}(x:y:z) = (\xi x: y: z)$ Again, we have $Y/N \cong \PP^1$. However, this time the stabilizer of the point $(-1:0:1)$ is trivial, whereas 
$(0:1:1)$ and $(0:-1:1)$ belong to different orbits. 
The  stabilizer of each point  $(0:1:1)$, $(0:-1:1)$  and $(0:1:0)$ is the group $N$ itself. Therefore, we have exact equivalences of triangulated categories 
$$
D^b\bigl(\Coh^N(Y)\bigr) \simeq D^b\bigl(\Pi_{(3, 3, 3)}\mathsf{-mod}\bigr)  \simeq D^b\bigl(\Sigma_{(3, 3, 3)}\mathsf{-mod}\bigr).
$$

\smallskip
\noindent 
(III) Now, let   $Y = \mathsf{Proj}\bigl(\kk[x,y,z]/(zy^2 - x^3 + x z^2)\bigr)$ and $G = \left\langle \varrho \,\big|\, \varrho^4 = e\right\rangle \cong \ZZ_4$. Then $G$ acts on $Y$ by the rule $\varrho(x:y:z) = (- x: i y: z)$ and $Y/G \cong \PP^1$. The stabilizer of the point $(1:0:1)$ is $\ZZ_2$, whereas the stabilizer of $(0:0:1)$ and $(0:1:0)$ is the group $G$ itself. Therefore, we have exact equivalences of triangulated categories 
$$
D^b\bigl(\Coh^G(Y)\bigr) \simeq D^b\bigl(\Pi_{(2, 4, 4)}\mathsf{-mod}\bigr)  \simeq D^b\bigl(\Sigma_{(2, 4, 4)}\mathsf{-mod}\bigr).
$$
\end{example}

\section{Tilting on real curve orbifolds}
In this section, we shall discuss some interesting 
and natural actions \emph{over $\RR$} on \emph{complex} projective curves. Do this, we begin with the local case. 

\begin{proposition}\label{P:ActionTypes} Let $G$ be a finite group,   $A = \CC\llbracket z\rrbracket$, $\idm = (z)$ and 
$G \stackrel{\phi}\lar \Aut_{\RR}(A)$ be an injective group homomorphism. Then the following two cases can occur. 
\begin{enumerate}
\item[(a)] For any $g \in G$ the homomorphism $A \stackrel{\phi_g}\lar A$ is $\CC$-linear. Then $G = \langle \varrho \,\big|\, \varrho^n = e\rangle$ is a cyclic group and there exists another choice of a local parameter $w \in \idm$ such that 
$\phi_\varrho(w) = \xi w$, where $\xi = \exp\left(\dfrac{2\pi i}{n}\right)$.
\item[(b)] Otherwise, 
\begin{equation}\label{E:DihedralPresentation}
G \cong D_n = \bigl\langle \sigma, \varrho \,\big|\, \sigma^2 = e = \varrho^n, \sigma \varrho \sigma^{-1} = \varrho^{-1}\bigr\rangle
\end{equation}
is a dihedral group for some $n\in \NN$. Moreover, there exists a choice of a local parameter $w \in \idm$ such that
\begin{equation}\label{E:action}
\left\{
\begin{array}{c}
\phi_{\sigma}(\alpha) = \bar{\alpha} \; \mbox{\rm for} \; \alpha \in \CC \; \mbox{\rm and} \; 
\phi_{\sigma}(w) = w \\
\; \, \phi_{\varrho}(\alpha) \, = \alpha \; \, \mbox{\rm for}\;  \alpha \in \CC \; \mbox{\rm and} \; 
\phi_{\varrho}(w) = \xi w,
\end{array}
\right.
\end{equation}
where $\xi = \exp\left(\dfrac{2\pi i}{n}\right)$.
\end{enumerate}
\end{proposition}
\begin{proof} First note that 
$
\bigl\{a \in A \,\big|\, a^2 +1 = 0\bigr\} = \{i, -i\}.
$
Since for any $g \in G$ the map $A \stackrel{\phi_g}\lar A$ is an automorphism of $\RR$-algebras, we conclude that $\phi_g(i) = \pm i$. Hence, any $\phi_g$ is either $\CC$-linear or $\CC$-antilinear. We put
$$
N := \left\{g \in G \, \big| \, \phi_g \; \mbox{is} \; \CC\mbox{\rm -linear}\right\}.
$$
By Lemma \ref{L:ActionDVR} we have: $N = \langle \varrho \, \big|\, 
\varrho^n = e\rangle \cong \ZZ_n$ for some $\varrho \in N$ and $n = |N|$. Moreover, there exists 
a local parameter $w \in \idm$ such that $\phi_\varrho(w) = \xi w$, where $\xi = \exp\left(\dfrac{2\pi i}{n}\right)$. The same proof allows one to construct $w \in \idm$ such that $\phi_g(w) = \xi_g w$ for any $g \in G$, where $\xi_g \in \CC^\ast$. 

If $N = G$ then we are done and have the case (a). Now assume that there exists $\sigma \in G \setminus N$. Then $\sigma^2 \in N$ and $\phi_\sigma(\alpha) = \bar{\alpha}$ for any $\alpha \in \CC$. Moreover, for any $g \in G \setminus N$ we have: $g \sigma \in N$. Hence, the elements $\varrho$ and $\sigma$ generate the group $G$. 

 We know that $\phi_\sigma(w) = \alpha w$ for some $\alpha \in \CC$ such that $|\alpha|^2 = 1$. Let $\zeta \in \CC^\ast$ be such that $\zeta^2 = \alpha$. Then $\phi_\sigma(\zeta w) = \bar\zeta \alpha w = \zeta w$. Replacing $w$ by $\zeta w$ we obtain: 
\begin{equation*}
\left\{
\begin{array}{c}
\; \; \, \phi_{\varrho}(\alpha) \, = \alpha \; \, \mbox{\rm for}\;  \alpha \in \CC \; \mbox{\rm and} \; 
\phi_{\varrho}(w) = \xi w, \\
\phi_{\sigma}(\alpha) = \bar{\alpha} \; \mbox{\rm for} \; \alpha \in \CC \; \mbox{\rm and} \; \phi_{\sigma}(w) = w. \\
\end{array}
\right.
\end{equation*}
The last formula implies that $\phi_{\sigma^2} = \id$. Since $\phi$ is injective, we conclude that $\sigma^2 = e$.  Analogously, we have $\phi_{\sigma \varrho} = 
\phi_{\varrho^{-1} \sigma}$, hence $\sigma \varrho = \varrho^{-1} \sigma$ and $G$ is a dihedral group. 
\end{proof}

\begin{lemma}\label{L:DihedralReduction} Let $G = D_n$ be the dihedral group given by the presentation (\ref{E:DihedralPresentation}), $N = \langle \varrho \rangle \cong \ZZ_n$ and $C = \langle \sigma\rangle \cong \ZZ_2$. Let $A$ be a ring and $G \stackrel{\phi}\lar \Aut(A)$ be a group homomorphism. Then the following results are true. 
\begin{enumerate}
\item[(a)] We have a group homomorphism $C \stackrel{\psi}\lar \Aut\bigl(A[N, \phi]\bigr)$, where
\begin{equation}\label{E:Action}
\psi_\sigma\bigl(a[h]\bigr) = \phi_\sigma(a)[h^{-1}] \; \mbox{\rm for any} \; a\in A, h\in N.
\end{equation}
\item[(b)]  There is a ring isomorphism
\begin{equation}\label{E:DihedralAction}
\bigl(A[N, \phi]\bigr)[C, \psi] \cong A[G, \phi], \; (a[h])\{\sigma^m\} \mapsto a[h\sigma^m] \; \mbox{\rm for any} \; a\in A, h\in N \;\mbox{\rm and} \; m \in \NN.
\end{equation} 
\end{enumerate}
Moreover, if $\kk$ is a field and $G$ acts on $A$ by $\kk$-algebra automorphisms then the action $\psi$ is also $\kk$-linear and (\ref{E:DihedralAction}) is an isomorphism of $\kk$-algebras. 
\end{lemma}

\smallskip
\noindent
\emph{Comment to the proof}. Both results can be verified by a straightforward computation and are therefore left to an interested reader as an exercise. \qed 

\begin{proposition}\label{P:SkewProducts} For any $n\in \NN$, let $G = D_n$ be the corresponding dihedral group acting on $A = \CC\llbracket z\rrbracket$ by $\RR$-algebra homomorphisms given by the formula (\ref{E:action}). Then we have an isomorphism of $\RR$-algebras
\begin{equation}
A[G, \phi] \cong M_2\bigl(H_n(O)\bigr),
\end{equation}
where $O= \RR\llbracket t^n\rrbracket$. 
\end{proposition}

\begin{proof} By Lemma \ref{L:DihedralReduction} we have: $A[G, \phi] \cong \bigl(A[N, \phi]\bigr)[C, \psi]$. Recall that 
we have an isomorphism of $\CC$-algebras $A[N, \phi] \stackrel{\mu}\lar \widehat{\CC\bigl[\vec{C}_n\bigr]}$ given by the formula (\ref{E:IsomorphismCyclicQuiver}). For any $\zeta \in \CC^\ast$ with $|\zeta| = 1$ we have: $\psi_{\sigma}(\zeta) = \bar{\zeta} = \zeta^{-1}$. For all $h \in N$ we have:  $\psi_\sigma([h]) = \bigl[h^{-1}\bigr]$. Hence, 
$$ 
\psi_{\sigma}\bigl(\varepsilon_k\bigr) = \psi_{\sigma}\left(\frac{1}{n} \sum\limits_{j=0}^{n-1} \zeta_k^j [\varrho^j]\right) = 
\sum\limits_{j=0}^{n-1} \zeta_k^{-j} [\varrho^{-j}] = \varepsilon_k
$$
for any $1 \le k \le n$. It follows that the induced action $\widehat{\CC\bigl[\vec{C}_n\bigr]} \stackrel{\psi_\sigma}\lar 
\widehat{\CC\bigl[\vec{C}_n\bigr]}
$ is given by the complex conjugation. 

\smallskip
\noindent
According to Lemma \ref{L:SkewGroupOfFields} we have: $\CC[C, \psi] \cong M_{2}(\RR)$, where $\CC \stackrel{\psi_\sigma}\lar \CC, \alpha \mapsto \bar{\alpha}$ is the complex conjugation. As a consequence, we get isomorphisms of $\RR$-algebras 
$$
A[G, \phi] \cong \widehat{\CC\bigl[\vec{C}_n\bigr]}\bigl[C, \psi\bigr] \cong M_2\bigl(H_n(O)\bigr),
$$
what proves the statement. 
\end{proof}

\begin{definition}\label{D:Orbipoints} Let $Y$ be a complete regular  curve over $\CC$, which we view as a scheme  over $\RR$.  Let 
$G \subseteq \Aut_{\RR}(Y)$ be a finite subgroup and $Y \stackrel{\pi} \lar Y/G =: X$ be the canonical projection. For $y \in Y_\circ$ let $N\subseteq G$ be the corresponding stabilizer group, $A = \widehat{\kO}_y$  and $x = \pi(y)$. We suppose that $|N| \ge 2$. 
\begin{enumerate}
\item[(a)] Assume that $N$ acts on $A$ by $\CC$-linear automorphisms. Then we say that $x \in X$ has type $n$ for $n = |N|$ (note that according to Proposition \ref{P:ActionTypes} we have $G \cong \ZZ_n$). 
\item[(b)] Assume that $N$ contains an element which acts as the complex conjugation on $A$. Then $N \cong D_n$   for some $n \in N$ (see again Proposition \ref{P:ActionTypes}) and we say that $x$ has type $\bar{n}$ provided $n \ge 2$.
\end{enumerate}
\end{definition}

\begin{remark} In the notation of Definition \ref{D:Orbipoints}, let $\XX = 
Y \hspace{-0.5mm}\sslash \hspace{-0.5mm}G = (X, \kH)$ be the corresponding non-commutative hereditary curve. Then points of $X$  of types $n$ and $\bar{n}$ for $n \in \NN_{\ge 2}$ are precisely those ones for which the order $\widehat{\kH}_x$ is not maximal; see Proposition \ref{P:SkewProducts}.
\end{remark}

\begin{remark} There are precisely three pairwise non-isomorphic real projective curves of genus zero:
\begin{enumerate}
\item[(a)] The real projective line $X_{\mathsf{re}} = \PP^1_{\RR}$. The corresponding tame bimodule $\Lambda_{\mathsf{re}}$ (see (\ref{E:TameBimodule})) is the path algebra of the Kronecker quiver:
$$
\Lambda_{\mathsf{re}} =
\RR\Bigl[\xymatrix{
\bu  \ar@/^/[r] \ar@/_/[r]  & \bu
}
\Bigr] \cong \left(
\begin{array}{cc}
\RR & \RR \oplus \RR \\
0   & \RR\\
\end{array}
\right).
$$
\item[(b)] The complex  projective line $X_{\mathsf{co}} = \PP^1_{\CC}$. The corresponding tame bimodule $\Lambda_{\mathsf{co}}$  is the path algebra of the Kronecker quiver over $\CC$:
$$
\Lambda_{\mathsf{co}} =
\CC\Bigl[\xymatrix{
\bu  \ar@/^/[r] \ar@/_/[r]  & \bu
}
\Bigr] \cong \left(
\begin{array}{cc}
\CC & \CC \oplus \CC \\
0   & \CC\\
\end{array}
\right).
$$
\item[(c)] The real conic $X_{\mathsf{qt}} = \mathsf{Proj}\bigl(\RR[x,y,z]/(x^2 + y^2 +z^2)\bigr)$. The corresponding tame bimodule is 
$$
\Lambda_{\mathsf{qt}} =
 \left(
\begin{array}{cc}
\RR & \HH \\
0   & \HH\\
\end{array}
\right),
$$
see \cite[Proposition 7.5]{Lenzing}.
\end{enumerate}
\end{remark}

\smallskip
\noindent
Let $Y'$ be a complete geometrically integral  regular   curve over $\RR$ and 
$$Y = \Spec{(\CC)} \times_{\Spec(\RR)} Y'.$$ Then the Galois group
$\mathsf{Gal}(\CC/\RR) = \langle \sigma \, \big| \, \sigma^2 = e\bigr\rangle$ canonically acts on $Y$ viewed as a scheme over $\RR$. In all examples below $\sigma$ acts as the complex conjugation.

Analogously to  Example \ref{E:EllipticQuotients1} and Example \ref{E:EllipticQuotients2}, we can consider finite group actions  on \emph{complex} elliptic curves viewed as schemes over $\RR$.

\begin{example}
Let $Y_\lambda = \mathsf{Proj}\bigl(\CC[x,y,z]/(y^2 z - (x-\lambda z)(x^2 + z^2))\bigr)$ for some $\lambda \in \RR$. Then the dihedral group $G = D_2 = \langle \sigma, \varrho\rangle \cong \ZZ_2 \times \ZZ_2$ acts on $Y_\lambda$ by the rule
$(x: y: z) \stackrel{\varrho}\mapsto (x: -y: z)$. The fixed points of this action are $(\lambda:0:1)$, $(0:i:1)$ and $(0:1:0)$ (note that $\sigma$ permutes $(0:i:1)$ and $(0:-i:1)$). The stabilizer of $(\lambda:0:1)$ and $(0:1:0)$ is the group $G$ itself, whereas the stabilizer of $(0:i:1)$ is 
$\langle \varrho\rangle \cong \ZZ_2$. 

We have: $Y_\lambda/G \cong X_{\mathsf{re}}$. Moreover, one can naturally choose homogeneous coordinates $(u: v)$ on $X_{\mathsf{re}}  = \mathsf{Proj}\bigl(\RR[u,v]\bigr)$ such that for the canonical projection $Y_{\lambda} \stackrel{\pi}\lar X_{\mathsf{re}}$ we have: 
$\pi(\lambda:0:1) = (\lambda:1)$ and $\pi(0:1:0) = (1:0)$. The point $ o = \pi(i:0:1) \in X_{\mathsf{re}}$ corresponds to the homogeneous ideal $u^2 + v^2 \in \RR[u,v]$. 

The above discussion shows that the corresponding non-commutative hereditary curve $\XX$ is of type $\bigl(X_{\mathsf{re}}, (2, \bar{2},  \bar{2})\bigr)$. More precisely, $\XX$ has 
\begin{enumerate}
\item[(a)] One special complex point $o$ of weight $2$,
\item[(b)] Two special real  points $(\lambda: 1)$ and $(1: 0)$ of weight $2$.
\end{enumerate}
We have an exact equivalence of triangulated categories
 \begin{equation*}
 D^b\bigl(\Coh^G(Y_\lambda)\bigr) \lar D^b(\Pi_{Y_\lambda, G}\mathsf{-mod})
 \end{equation*}
 for an appropriate squid algebra $\Pi_{Y_\lambda, G}$ of the form (\ref{E:Squid}).
\end{example}

\begin{example} Let $Y = \mathsf{Proj}\left(\CC[x, y,z]/(zy^2 - x^3 - z^3)\right)$ and $G = \left\langle \sigma, \varrho \right\rangle \cong D_6$. Then $G$ acts on $Y$ by the rule $\varrho(x:y:z) = (\xi x: -y: z)$, where $\xi = \exp\left(\dfrac{2\pi i}{3}\right)$. The special orbits of the $G$-action are those of 
\begin{enumerate}
\item[(a)]  the point $(-1:0:1)$, whose stabilizer  is $D_2$;
\item[(b)]  the point $(0:1:1)$, whose stabilizer  is $D_3$; 
\item[(c)]  the point $(0:1:0)$, whose stabilizer  is $D_6$.
\end{enumerate}
The corresponding hereditary curve $\XX$ has type $\bigl(X_{\mathsf{re}}, (\bar{2}, \bar{3},  \bar{6})\bigr)$. Since the group $\Aut_{\RR}(X_{\mathsf{re}})$ acts transitively on triples of distinct closed real points of $X_{\mathsf{re}}$, we may assume that the special points of $\XX$ are $(0:1)$, $(1:0)$ and $(1: 1)$, respectively. 

Now, let $\widetilde\varrho = \varrho^4$. Consider the subgroup $D_3 \cong N = \langle \sigma, \widetilde\varrho\rangle \subset G$. Then $N$ acts on $Y$ by the rule $\widetilde{\varrho}(x:y:z) = (\xi x: y: z)$. Again, we have $Y/N \cong X_{\mathsf{re}}$. The points $(0:1:1)$, $(0:-1:1)$ and $(0:1:0)$ are stabilized by $N$. Hence, the corresponding hereditary curve $\XX$ hat type 
$\bigl(X_{\mathsf{re}}, (\bar{3}, \bar{3},  \bar{3})\bigr)$.
\end{example}

\begin{example} Consider now  
$Y = \mathsf{Proj}\left(\CC[x, y,z]/(zy^2 - x^3 + z^3)\right)$ and $D_3 \cong N = \langle \sigma, \widetilde\varrho\rangle$, where  $\widetilde{\varrho}(x:y:z) = (\xi x: y: z)$ for $\xi = \exp\left(\dfrac{2\pi i}{3}\right)$. Again, we have $Y/N \cong X_{\mathsf{re}}$.  However, this time $\sigma(0:i:1) = (0:-i:1)$. As a consequence, we now have only two special orbits of the $N$-action on $Y$:
\begin{enumerate}
\item[(a)] Those of $(0:i:1)$ whose stabilizer is $\ZZ_3$
\item[(b)] Those of $(0:1:0)$ whose stabilizer is $D_3$
\end{enumerate}
As a consequence, the corresponding hereditary curve $\XX$ has type 
$\bigl(X_{\mathsf{re}}, (3,  \bar{3})\bigr)$.
\end{example}

\begin{example} Let $A = \CC[x, y]/(y^2 + (x^2 + \lambda)^2 +1)$ for some $\lambda \in \RR$ and $Y = Y_\lambda$ be the smooth regular projective curve over $\CC$ with is the completion of $\breve{Y} = \Spec(A) \subset \mathbbm{A}^2_\CC$. The dihedral group $D_2 = \langle \sigma, \varrho\rangle$ operates on $A$ by the rule
$x \stackrel{\varrho}\mapsto  -x$, $y \stackrel{\varrho}\mapsto y$.  It is clear that this action on $\Spec(A)$ can be extended to an action on $Y$. Since $A^G = \RR[w, y]/(z^2 + y^2 +1)$ for $w = x^2 + \lambda$, we may conclude that $Y/G \cong X_{\mathsf{qt}}$. 

The action of $G$ on $Y$ has two special orbits. 
 The first one is the orbit of the point $(0, i \sqrt{1+\lambda^2}) \in \breve{Y}$. The corresponding stabilizer is $\langle\varrho \rangle \cong \ZZ_2$.

\smallskip
\noindent
To describe the second orbit, consider the closure $\bar{Y}$ of $Y$ in $\PP^2_\CC$. We have: $\bar{Y} = 
\mathsf{Proj}\bigl(\CC[x, y, z]/(y^2 z^2 + (x^2 + \lambda z^2)^2 + z^4)\bigr)$. Note that the point $o = (0:1:0) \in \bar{Y}$  is singular. The curve $Y$ is the normalization of $\bar{Y}$. Let $Y \stackrel{\nu}\lar \bar{Y}$ be the normalization map. Then $\nu^{-1}(o) = \left\{o_+, o_-\right\}$  and $\sigma(o_\pm) = o_{\mp}$. A straightforward local computation shows that 
the stabilizer of $o_+$ is $\langle\varrho\rangle \cong \ZZ_2$. 
It follows  that the corresponding hereditary curve $\XX$ has type 
$\bigl(X_{\mathsf{qt}}, (2, 2)\bigr)$.
\end{example}

A systematic way to construct finite group actions on complex elliptic curves viewed as real algebraic schemes comes from wallpaper groups. To explain this construction, recall that a \emph{Klein surface} $\dX$ is a \emph{dianalytic manifold} (possibly, with non-empty boundary)  of complex dimension one; see \cite{AllingGreenleafFirst, AllingGreenleaf, BEGG} for the details. Klein surfaces naturally form a category. An important result due to Alling and Greenleaf
 asserts that the category of compact Klein surfaces is equivalent 
to the category of regular complete curves  over $\RR$; see 
\cite[Theorem 3]{AllingGreenleafFirst}, 
\cite[Section II.3]{AllingGreenleaf} as well as \cite[Appendix A]{BEGG} for further elaborations. The key point is the following: the set $\mathbbm{M}(\dX)$ of all meromorphic functions on a connected Klein surface $\dX$  is an algebraic function field of one variable over $\RR$ (i.e.~a finitely generated field extension of $\RR$ of transcendence degree one); see \cite[Theorem 1]{AllingGreenleafFirst} as well as 
\cite{AllingGreenleaf}. The field $\mathbbm{M}(\dX)$ defines a uniquely determined (up to isomorphisms) regular projective curve $X$ over $\RR$. The main point is to prove that the correspondence $\dX \mapsto \mathbbm{M}(\dX)$ defines a contravariant equivalence between the category of connected Klein surfaces and the category of real algebraic function fields in one variable. 

\smallskip
\noindent
 In particular, in genus zero we have: 
\begin{enumerate}
\item[(a)] the closed disc $\mathfrak{D} = \left\{z \in \CC \, \big|\, |z| \le 1\right\}$ has the function field $\RR(z)$ and corresponds to the curve $X_{\mathsf{re}}$;
\item[(b)] the Riemann sphere $\mathfrak{S}$  has the function field $\CC(z)$ and corresponds to  $X_{\mathsf{co}}$;
\item[(c)] the real projective plane  $\mathfrak{P}$
 has the function field $\RR(y)[x]/(x^2 + y^2 +1)$ and corresponds to the 
curve $X_{\mathsf{qt}}$.  
\end{enumerate}

Recall that the Euclidean group $\mathsf{E}_2 = \mathsf{O}_2(\RR) \ltimes \RR^2$ is the group of isometries of the Euclidean plane $\RR^2 = \CC$. For any $(A, \vec{v}) \in \mathsf{E}_2$ we have the corresponding automorphism 
$$
\RR^2 \lar \RR^2, \vec{x} \mapsto A \vec{x} + \vec{v},
$$
which is either analytic (if $\det(A) = 1$) or anti-analytic (if $\det(A) = -1$) with respect to the standard complex structure on $\RR^2 = \CC$. 

A \emph{wallpaper group} $W$ (also called plane crystallographic group) is a discrete cocompact subgroup of  $\mathsf{E}_2$;
see  for example \cite{Klemm, Montesinos}. Let $T$ be the subgroup of $W$ consisting of all translations. Bieberbach's Theorem asserts that  $T \lhd W$ is a normal subgroup, $T \cong \ZZ^2$                                                         
and $G:= W/T \subset \mathsf{O}_2(\RR)$ is a finite group (called \emph{point group} of $W$). Obviously, 
$\dY  = \CC/T$ is a complex torus and the point group $G$ acts on $\mathfrak{Y}$ by dianalytic automorphisms. The quotient $\dX_W = \RR^2/W =  \dY/G$ is a compact flat surface orbifold; see \cite[Appendix A.3]{Montesinos}.

Let  $\mathfrak{Z}$ be a surface orbifold and $p \in \mathfrak{Z}$ be its singular point. Then $p$  belongs to a one of the following three classes:
\begin{enumerate}
\item[(a)] Mirror point, if it admits a neighbourhood isomorphic to $\RR^2/\ZZ_2$, where the generator of $\ZZ_2$ acts by a reflection (say, with respect to  the x-axis)
\item[(b)] Elliptic point of order $n\in \NN_{\ge 2}$ (denoted by $n$), if it admits a neighbourhood isomorphic to $\RR^2/\ZZ_n$, where $\ZZ_n$ acts on $\RR^2$ by rotations.
\item[(c)] Corner reflector point of order $n\in \NN_{\ge 2}$ (denoted by $\bar{n}$), if it admits a neighbourhood isomorphic to 
$\RR^2/D_n$  with respect to the natural action of the dihedral group on $\RR^2$. 
\end{enumerate}
If $p \in \dX_W$ is a mirror point then it is just an ordinary point of the boundary of $\dX_W$. An essential information about $\dX_W$ (viewed as an surface orbifold) is governed by its diffeomorphism type and  by the number/position  of its elliptic and corner reflector points. 

Let $\mathbbm{M}$ be the field of meromorphic functions on $\dY$. Then we have a natural  group embedding 
$G \subset  \Aut_{\RR}(\mathbbm{M})$ induced by the action of $G$ on $\dY$ (viewed as a Klein surface). Let $Y$ be the complex elliptic curve corresponding to $\dY$. Then we have a group embedding $G \subset  \Aut_{\RR}(Y)$. Let $X = Y/G$ and 
$\XX = \XX_W = Y \hspace{-1mm}\sslash \hspace{-1mm} G$ be the corresponding hereditary curve. The key Proposition \ref{P:ActionTypes} as well as the aforementioned Alling--Greenleaf equivalence of categories allows one to relate the datum $(X, \rho)$ defining $\XX$ with the orbifold notation  of the underlying wallpaper group $W$. 

\begin{theorem}\label{T:Wallpapers}
Let $W$ be a wallpaper group for which  $g(X) = 0$. Then there exists a real squid  algebra $\Pi_W$ of \emph{tubular type} and an exact equivalence of triangulated categories 
\begin{equation}\label{E:Wallpaper}
D^b\bigl(\Coh(\XX_W)\bigr) \simeq D^b\bigl(\Pi_W\mathsf{-mod}\bigr).
\end{equation}
\end{theorem}

\begin{proof} Since $g(X) = 0$, Theorem \ref{T:TiltingMain} implies that there exists a squid algebra $\Pi_W$ such that $D^b\bigl(\Coh(\XX_W)\bigr) \simeq D^b\bigl(\Pi_W\mathsf{-mod}\bigr)$. Recall (see \cite{Klemm, Montesinos}) the classification of the isomorphism classes of wallpaper groups and the corresponding flat surface orbifolds:

\begin{center}
\begin{tabular}{|c|c|c|c|}
\hline
\textnumero & Wallpaper group & Orbifold type & hereditary  curve type\\
\hline
1 & hexatrope group & $\mathfrak{S}(2, 3, 6)$ & $X_{\mathsf{co}}(2, 3, 6)$ \\
2 & tetratrope  group & $\mathfrak{S}(2, 4, 4)$ & $X_{\mathsf{co}}(2, 4, 4)$ \\
3 & tritrope group & $\mathfrak{S}(3, 3, 3)$  & $X_{\mathsf{co}}(3, 3, 3)$ \\
4 & ditrope group & $\mathfrak{S}(2, 2, 2, 2)$ & $X_{\mathsf{co}}(2, 2, 2, 2)$ \\
5 & hexascope  group & $\mathfrak{D}(\bar{2}, \bar{3}, \bar{6})$ & $X_{\mathsf{re}}(\bar{2}, \bar{3}, \bar{6})$\\
6 & tetrascope  group & $\mathfrak{D}(\bar{2}, \bar{4}, \bar{4})$ &  $X_{\mathsf{re}}(\bar{2}, \bar{4}, \bar{4})$\\
7 & triscope group & $\mathfrak{D}(\bar{3}, \bar{3}, \bar{3})$ & $X_{\mathsf{re}}(\bar{3}, \bar{3}, \bar{3})$ \\
8 & discope group & $\mathfrak{D}(\bar{2}, \bar{2}, \bar{2}, \bar{2})$ & 
$X_{\mathsf{re}}(\bar{2}, \bar{2}, \bar{2}, \bar{2})$
\\
9 & tetragyro group & $\mathfrak{D}(4, \bar{2})$ & $X_{\mathsf{re}}(4, \bar{2})$\\
10 & trigyro group & $\mathfrak{D}({3}, \bar{3})$ & $X_{\mathsf{re}}(3, \bar{3})$ \\
11 & digyro group & $\mathfrak{D}({2}, 2)$ & $X_{\mathsf{re}}(2, {2})$ \\
12 & dirhomb group & $\mathfrak{D}({2}, \bar{2}, \bar{2})$ & $X_{\mathsf{re}}({2}, \bar{2}, \bar{2})$ \\
13 & diglide group & $\mathfrak{P}({2}, 2)$ & $X_{\mathsf{qt}}({2}, 2)$ \\
\hline
\hline
14 & monotrope group & torus & $\mathsf{Proj}\bigl(\CC[x,y,z]/(zy^2 -x^3 +x z^2)\bigr)$\\
15  & monoglide group & Klein bottle & $\mathsf{Proj}\bigl(\RR[x,y,z]/(z^2 y^2 + (x^2 + z^2)(x^2 + z^2))\bigr)$ \\
16  & monorhomb group & M\"obius band & $\mathsf{Proj}\bigl(\RR[x,y,z]/(zy^2 -x^3 - x z^2)\bigr)$ \\
17 & monoscope group & annulus &  $\mathsf{Proj}\bigl(\RR[x,y,z]/(zy^2 -x^3 +x z^2)\bigr)$ \\
\hline
\end{tabular}
\end{center}
The last four types of the above table correspond to real projective curve of genus one, the stated correspondence is taken from \cite[Example 1]{AllingGreenleaf}. The corresponding derived category $D^b\bigl(\Coh(\XX)\bigr)$ does not have tilting objects. In the first thirteen cases, it  follows from the stated classification, that the  squid algebra $\Pi_W$ has a tubular type.
\end{proof}

\begin{remark}\label{R:LenzingKussin} The correspondence between wallpaper groups and real hereditary curves of tubular type was for the first time observed by Lenzing many years ago \cite{LenzingLastTalk}. Kussin in \cite[Corollary 13.23]{Kussin} gave a classification of all hereditary curves of tubular type. From this classification it  became apparent that the curves
of type $\XX_W$ are precisely those ones, for which $[\eta_{\XX}] = 0$: indeed,  the corresponding numerical patterns are the same.  Kussin informed me about another approach to establish a more concrete correspondence between wallpaper groups and exceptional hereditary curves of tubular type \cite{KussinNext}. However, the works
\cite{Kussin, KussinNext} are heavily based on the ``axiomatic approach''  to non-commutative hereditary curves and the corresponding proofs are technically different from the ones given in this paper.
\end{remark}

\end{document}